\documentclass[a4paper,11pt]{amsart}

\usepackage[left=1.3in,right=1.3in,top=1.1in,bottom=1.1in,includehead,includefoot]{geometry}
\usepackage[OT4]{fontenc}
\usepackage{amsthm,amsmath,amssymb}
\usepackage{enumitem}
\usepackage{tikz}
\usetikzlibrary{patterns}
\usepackage[pdftitle={Coloring intersection graphs of arc-connected sets in the plane},
            pdfauthor={M. Lason, P. Micek, A. Pawlik, B. Walczak}]{hyperref}

\renewenvironment{enumerate}{\begin{enumorig}[label=\textup{(\roman*)}, noitemsep, topsep=2pt plus 2pt, labelindent=.2em, leftmargin=*, widest=iii]}{\end{enumorig}}

\renewenvironment{itemize}{\begin{itemorig}[label=\textbullet, noitemsep, topsep=2pt plus 2pt, labelindent=.5em, labelsep=.5em, leftmargin=*]}{\end{itemorig}}

\newtheorem{theorem}{Theorem}
\newtheorem{proposition}[theorem]{Proposition}
\newtheorem{lemma}[theorem]{Lemma}
\newtheorem{claim}[theorem]{Claim}
\newtheorem{corollary}[theorem]{Corollary}
\theoremstyle{definition}
\newtheorem{problem}{Problem}

\let\leq\leqslant
\let\geq\geqslant
\let\setminus\smallsetminus
\def\tbigcup{\textstyle\bigcup\displaystyle}

\def\setN{\mathbb{N}}
\def\setR{\mathbb{R}}

\def\calB{\mathcal{B}}
\def\calC{\mathcal{C}}
\def\calD{\mathcal{D}}
\def\calF{\mathcal{F}}
\def\calG{\mathcal{G}}
\def\calH{\mathcal{H}}
\def\calK{\mathcal{K}}
\def\calM{\mathcal{M}}
\def\calN{\mathcal{N}}
\def\calR{\mathcal{R}}
\def\calS{\mathcal{S}}
\def\calX{\mathcal{X}}
\def\calY{\mathcal{Y}}
\def\calZ{\mathcal{Z}}

\let\int\undefined
\DeclareMathOperator{\base}{base}
\DeclareMathOperator{\int}{int}
\DeclareMathOperator{\ext}{ext}
\DeclareMathOperator{\cut}{cut}
\DeclareMathOperator{\leftclip}{leftclip}
\DeclareMathOperator{\rightclip}{rightclip}

\makeatletter
\let\old@setaddresses\@setaddresses
\def\@setaddresses{\bgroup\parindent 0pt\let\scshape\relax\old@setaddresses\egroup}
\makeatother

\title{Coloring intersection graphs of arc-connected sets in the plane}

\author{Micha\l{} Laso\'n}
\author{Piotr Micek}
\author{Arkadiusz Pawlik}
\author{Bartosz Walczak}

\thanks{Journal version of this paper appeared in \emph{Discrete Comput.\ Geom.}, 52(2):399--415, 2014.}
\thanks{M.~Laso\'n, P.~Micek, and A.~Pawlik were supported and B.~Walczak was partially supported by Ministry of Science and Higher Education of Poland grant 884/N-ESF-EuroGIGA/10/2011/0 within ESF EuroGIGA project GraDR\@. B.~Walczak was partially supported by Swiss National Science Foundation grant 200020-144531.}

\address[Micha\l{} Laso\'n]{Theoretical Computer Science Department, Faculty of Mathematics and Computer Science, Jagiellonian University, Krak\'ow, Poland; Institute of Mathematics of the Polish Academy of Sciences, Warsaw, Poland; \'Ecole Polytechnique F\'ed\'erale de Lausanne, Switzerland}
\email{michalason@gmail.com}

\address[Piotr Micek, Arkadiusz Pawlik]{Theoretical Computer Science Department, Faculty of Mathematics and Computer Science, Jagiellonian University, Krak\'ow, Poland}
\email{micek@tcs.uj.edu.pl, pawlik@tcs.uj.edu.pl}

\address[Bartosz Walczak]{Theoretical Computer Science Department, Faculty of Mathematics and Computer Science, Jagiellonian University, Krak\'ow, Poland; \'Ecole Polytechnique F\'ed\'erale de Lausanne, Switzerland}
\email{walczak@tcs.uj.edu.pl}

\begin{document}
\baselineskip 14pt

\begin{abstract}
A family of sets in the plane is simple if the intersection of its any subfamily is arc-connected, and it is pierced by a line $L$ if the intersection of its any member with $L$ is a nonempty segment. 
It is proved that the intersection graphs of simple families of compact arc-connected sets in the plane pierced by a common line have chromatic number bounded by a function of their clique number.
\end{abstract}

\maketitle

\section{Introduction}

A \emph{proper coloring} of a graph is an assignment of colors to the vertices of the graph such that no two adjacent ones are assigned the same color.
The minimum number of colors sufficient to color a graph $G$ properly is called the \emph{chromatic number} of $G$ and denoted by $\chi(G)$.
The maximum size of a clique (a set of pairwise adjacent vertices) in a graph $G$ is called the \emph{clique number} of $G$ and denoted by $\omega(G)$.
It is clear that $\chi(G)\geq\omega(G)$.

The chromatic and clique numbers of a graph can be arbitrarily far apart. 
There are various constructions of graphs that are triangle-free (have clique number $2$) and still have arbitrarily large chromatic number. 
The first one was given in 1949 by Zykov \cite{Zyk49}, and the one perhaps best known is due to Mycielski \cite{Myc55}.
However, these classical constructions require a lot of freedom in connecting vertices by edges, and many important classes of graphs derived from specific (e.g.\ geometric) representations have chromatic number bounded in terms of the clique number.
A class of graphs is called \emph{$\chi$-bounded} if there is a function $f\colon\setN\to\setN$ such that $\chi(G)\leq f(\omega(G))$ holds for any graph $G$ from the class.

In this paper, we focus on the relation between the chromatic number and the clique number for geometric intersection graphs.
The \emph{intersection graph} of a family of sets $\calF$ is the graph with vertex set $\calF$ and edge set consisting of pairs of intersecting elements of $\calF$.
We consider finite families $\calF$ of arc-connected compact sets in the plane which are \emph{simple} in the sense that the intersection of any subfamily of $\calF$ is also arc-connected.
We usually identify the family $\calF$ with its intersection graph and use such terms as chromatic number, clique number or $\chi$-boundedness referring directly to $\calF$.

In the one-dimensional case of subsets of $\setR$, the only arc-connected compact sets are closed intervals.
They define the class of \emph{interval graphs}, which have chromatic number equal to their clique number.
The study of the chromatic number of intersection graphs of geometric objects in higher dimensions was initiated in the seminal paper of Asplund and Gr\"unbaum \cite{AG60}, where they proved that the families of axis-aligned rectangles in $\setR^2$ are $\chi$-bounded.
On the other hand, Burling \cite{Bur65} showed that intersection graphs of axis-aligned boxes in $\setR^3$ with clique number $2$ can have arbitrarily large chromatic number.

Since then, a lot of research focused on proving $\chi$-boundedness of the families of geometric objects in the plane with various restrictions on the kind of objects considered, their positions, or the way they can intersect.
Pach \cite{Pac80} proved that for every $\delta>0$, the families of convex compact sets in the plane that are \emph{$\delta$-fat}, which means that the ratio between the area of the set and the area of its minimum bounding disc is at least $\delta$, are $\chi$-bounded.
Pach's result (with the area replaced by the volume and the disc replaced by the ball) actually holds in any Euclidean space $\setR^d$.
Gy\'arf\'as \cite{Gya85,Gya86} proved that the families of chords of a circle are $\chi$-bounded.
This was generalized by Kostochka and Kratochv\'{\i}l \cite{KK97} to the families of convex polygons inscribed in a circle.
Kim, Kostochka and Nakprasit \cite{KKN04} showed that every family $\calF$ of homothets (uniformly scaled copies) of a fixed convex compact set in the plane satisfies $\chi(\calF)\leq 6\omega(\calF)-6$, while every family $\calF$ of translates of a fixed convex compact set in the plane satisfies $\chi(\calF)\leq 3\omega(\calF)-2$.
Aloupis et al.\ \cite{ACC+09} proved that the families of \emph{pseudo-discs}, that is, closed disc homeomorphs in the plane the boundaries of any two of which cross at most twice, are $\chi$-bounded.
Note that families of convex sets or pseudo-discs are simple.
Fox and Pach \cite{FP10} showed that the intersection graphs of any arc-connected compact sets in the plane that do not contain a fixed bipartite subgraph $H$ have chromatic number bounded by a function of $H$.
The above-mentioned results of \cite{ACC+09,FP10,KKN04,Pac80} are actually stronger---they state that the number of edges of the intersection graph of a respective family $\calF$ is bounded by $f(\omega(\calF)){|\calF|}$ for some function $f$.

A family of sets $\calF$ is \emph{pierced} by a line $L$ if the intersection of any member of $\calF$ with $L$ is a nonempty segment.
McGuinness \cite{McG96} proved that the families of L-shapes (shapes consisting of a horizontal and a vertical segments of arbitrary lengths, forming the letter `L') pierced by a fixed vertical line are $\chi$-bounded.
Later \cite{McG00}, he showed that the triangle-free simple families of compact arc-connected sets in the plane pierced by a common line have bounded chromatic number.
Suk \cite{Suk14} proved $\chi$-boundedness of the simple families of $x$-monotone curves intersecting a fixed vertical line.
In this paper, we generalize the results of McGuinness, allowing any bound on the clique number, and of Suk, removing the $x$-monotonicity condition.

\begin{theorem}\label{thm:intro}
The class of simple families of compact arc-connected sets in the plane pierced by a common line is\/ $\chi$-bounded.
\end{theorem}

By contrast, Pawlik et al.\ \cite{PKK+13,PKK+14} proved that there are intersection graphs of straight-line segments (or geometric sets of many other kinds) with clique number $2$ and arbitrarily large chromatic number.
This justifies the assumption of Theorem \ref{thm:intro} that the sets are pierced by a common line.
The best known upper bound on the chromatic number of simple families of curves in the plane with clique number $\omega$ is $\smash[t]{\bigl(\frac{\log n}{\log\omega}\bigr)}^{\smash[t]{O(\log\omega)}}$ due to Fox and Pach \cite{FP12}.

The bound on the chromatic number following from our proof of Theorem \ref{thm:intro} is double exponential in terms of the clique number.

The ultimate goal of this quest is to understand the border line between the classes of graphs (and classes of geometric objects) that are $\chi$-bounded and those that are not.
In a preliminary version of this paper, we proposed the following two problems.

\begin{problem}
Are the families (not necessarily simple) of $x$-monotone curves in the plane pierced by a common vertical line $\chi$-bounded?
\end{problem}

\begin{problem}
Are the families of curves in the plane pierced by a common line $\chi$-bounded?
\end{problem}

Rok and Walczak \cite{RW14} proved recently that the answers to both these questions are positive.
However, the bound on the chromatic number in terms of the clique number resulting from their proof is enormous (greater than an exponential tower), which is much worse than the double exponential bound of Theorem \ref{thm:intro}.

\section{Topological preliminaries}

All the geometric sets that considered in this paper are assumed to be subsets of the Euclidean plane $\setR^2$ or, further in the paper, subsets of the closed upper halfplane $\setR\times[0,+\infty)$.
An \emph{arc} between points $x,y\in\setR^2$ is the image of a continuous injective map $\phi\colon[0,1]\to\setR^2$ such that $\phi(0)=x$ and $\phi(1)=y$.
A set $X\subset\setR^2$ is \emph{arc-connected} if any two points of $X$ are connected by an arc in $X$.
The union of two arc-connected sets that have non-empty intersection is itself arc-connected.
More generally, if $\calX$ is a family of arc-connected sets whose intersection graph is connected, then $\bigcup\calX$ is arc-connected.
For a set $X\subset\setR^2$, the relation $\{(x,y)\in X^2\colon X$ contains an arc between $x$ and $y\}$ is an equivalence, whose equivalence classes are the \emph{arc-connected components} of $X$.
Every arc-connected component of an open set is itself an open set.

All families of sets that we consider are finite.
A family $\calF$ of sets in $\setR^2$ is \emph{simple} if the intersection of any subfamily of $\calF$ is arc-connected (possibly empty).
A set $X$ is \emph{simple with respect to} a family $\calY$ if $\{X\}\cup\calY$ is simple.

\begin{lemma}\label{lem:curve}
Let\/ $X$ be a compact arc-connected set and\/ $\calY$ be a family of compact arc-connected sets such that\/ $X$ is simple with respect to\/ $\calY$ and the intersections of the members of\/ $\calY$ with\/ $X$ are pairwise disjoint.
Between any points\/ $x_1,x_2\in X$, there is an arc\/ $A\subset X$ that is simple with respect~to\/~$\calY$.
\end{lemma}

\begin{proof}
Let $\calY=\{Y_1,\ldots,Y_n\}$.
For $i\in\{0,\ldots,n\}$, we construct an arc $A_i\subset X$ between $x_1$ and $x_2$ that is simple with respect to $\{Y_1,\ldots,Y_i\}$.
As $X$ is arc-connected, we pick $A_0$ to be any arc between $x_1$ and $x_2$ within $X$.
We construct $A_i$ from $A_{i-1}$ as follows.
If $A_{i-1}\cap Y_i=\emptyset$, then we take $A_i=A_{i-1}$.
Otherwise, let $y_1$ and $y_2$ be respectively the first and the last points on $A_{i-1}$ that belong to $Y_i$ (which exist as $A_i\cap Y_i$ is non-empty and compact).
To obtain $A_i$, replace the part of $A_{i-1}$ between $y_1$ and $y_2$ by any arc between $y_1$ and $y_2$ in $X\cap Y_i$ (which exists because $X\cap Y_i$ is arc-connected).
Clearly, $A_i$ is simple with respect to $Y_i$.
Since $X\cap Y_i$ is disjoint from each of $Y_1,\ldots,Y_{i-1}$, $A_i$ remains simple with respect to $Y_1,\ldots,Y_{i-1}$.
\end{proof}

A \emph{Jordan curve} is the image of a continuous map $\phi\colon[0,1]\to\setR^2$ such that $\phi(0)=\phi(1)$ and $\phi$ is injective on $[0,1)$.
The famous Jordan curve theorem states that if $C\subset\setR^2$ is a Jordan curve, then $\setR^2\setminus C$ has exactly two arc-connected components, one bounded and one unbounded.
An extension of this, called Jordan-Sch\"onflies theorem, adds that there is a homeomorphism of $\setR^2$ that maps $C$ to a unit circle, the bounded arc-connected component of $\setR^2\setminus C$ to the interior of this circle, and the unbounded arc-connected component of $\setR^2\setminus C$ to the exterior of the circle.

We will use a special case of the Jordan curve theorem for arcs in the closed upper halfplane $\setR\times[0,+\infty)$.
Namely, if $x$ and $y$ are two points on the horizontal axis $\setR\times\{0\}$ and $A$ is an arc between $x$ and $y$ such that $A\setminus\{x,y\}\subset\setR\times(0,+\infty)$, then the set $(\setR\times[0,+\infty))\setminus A$ has exactly two arc-connected components, one bounded and one unbounded.
This in particular implies that for any four points $x_1$, $x_2$, $y_1$ and $y_2$ in this order on the horizontal axis, every arc in $\setR\times[0,+\infty)$ between $x_1$ and $y_1$ intersects every arc in $\setR\times[0,+\infty)$ between $x_2$ and $y_2$.

\section{Grounded families}

In Theorem \ref{thm:intro}, a family $\calF$ compact arc-connected sets in the plane is assumed to be pierced by a common line.
We assume without loss of generality that this piercing line is the horizontal axis $\setR\times\{0\}$
and call it the \emph{baseline}.
The \emph{base} of a set $X$, denoted by $\base(X)$, is the intersection of $X$ with the baseline.

We fix a positive integer $k$ and assume $\omega(\calF)\leq k$.
The intersection graph of the bases of the members of $\calF$ is an interval graph, so it can be properly colored with $k$ colors.
To find a proper coloring of $\calF$ with a number of colors bounded in terms of $k$, we can restrict our attention to one color class in the coloring of this interval graph.
Therefore, without loss of generality, we assume that no two members of $\calF$ intersect on the baseline and show that $\calF$ can be colored properly with a bounded number of colors.
Moreover, it is clear that the families $\calF^+=\{X\cap(\setR\times[0,+\infty))\colon X\in\calF\}$ and $\calF^-=\{X\cap(\setR\times(-\infty,0])\colon X\in\calF\}$ are simple.
It suffices to obtain proper colorings $\phi^+$ and $\phi^-$ of $\calF^+$ and $\calF^-$, respectively, with bounded numbers of colors, since then $\calF$ may be colored by pairs $(\phi^+,\phi^-)$.
We only focus on coloring $\calF^+$, as $\calF^-$ can be handled by symmetry.
To simplify notation, we rename $\calF^+$ to $\calF$.
Therefore, each set $X\in\calF$ is assumed to satisfy the following:
\begin{itemize}
\item $X\subset\setR\times[0,+\infty)$,
\item $X\cap\setR\times\{0\}$ is a non-empty interval,
\item $X$ is compact and arc-connected,
\end{itemize}
and $\calF$ is assumed to be simple.
Any set that satisfies the conditions above is called \emph{grounded}, and any simple family of grounded sets with pairwise disjoint bases is also called \emph{grounded} (see Figure \ref{fig:grounded-family}).

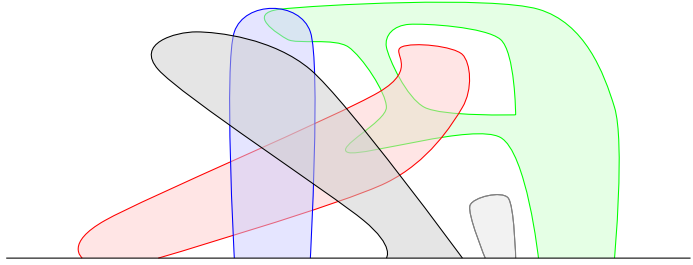
\begin{figure}[t]
\centering
\begin{tikzpicture}
\coordinate (Z0) at (4.0, 0.0);
\coordinate (Z1) at (3.5, 1.6);
\coordinate (Z2) at (1.5, 1.4);
\coordinate (Z3) at (2.0, 2.0);
\coordinate (Z32) at (1.5, 2.8);
\coordinate (Z34) at (0.7, 2.9);
\coordinate (Z35) at (0.6, 3.3);
\coordinate (Z36) at (4.0, 3.3);
\coordinate (Z4) at (5.0, 2.0);
\coordinate (Z5) at (5.0, 0.0);

\fill[opacity=0.5, fill=white!80!green] plot[smooth, tension=0.5] coordinates
  {(Z0) (Z1) (Z2) (Z3) (Z32) (Z34) (Z35) (Z36) (Z4) (Z5)};
\draw[color=green] plot[smooth, tension=0.5] coordinates
  {(Z0) (Z1) (Z2) (Z3) (Z32) (Z34) (Z35) (Z36) (Z4) (Z5)};

\coordinate (ZZ1) at (3.7, 1.9);
\coordinate (ZZ2) at (2.5, 2.0);
\coordinate (ZZ32) at (2.0, 2.8);
\coordinate (ZZ33) at (2.3, 3.1);
\coordinate (ZZ34) at (3.5, 2.9);
\coordinate (ZZ4) at (4.8, 1.8);

\draw[color=green, fill=white] plot[smooth, tension=0.5] coordinates
  {(ZZ1) (ZZ2) (ZZ32) (ZZ33) (ZZ34) (ZZ1)};

\coordinate (R0) at (-2.0, 0.0);
\coordinate (R05) at (-1.7, 0.5);
\coordinate (R1) at (1.9, 2.2);
\coordinate (R15) at (2.2, 2.8);
\coordinate (R2) at (3.0, 2.7);
\coordinate (R3) at (3.0, 2.0);
\coordinate (R4) at (2.0, 1.0);
\coordinate (R5) at (-1.0, 0.0);

\fill[opacity=0.5, fill=white!80!red] plot[smooth, tension=0.5] coordinates
  {(R0) (R05) (R1) (R15) (R2) (R3) (R4) (R5)};
\draw[color=red] plot[smooth, tension=0.5] coordinates
  {(R0) (R05) (R1) (R15) (R2) (R3) (R4) (R5)};

\coordinate (N0) at (0.0, 0.0);
\coordinate (N1) at (0.0, 3.0);
\coordinate (N2) at (1.0, 3.0);
\coordinate (N3) at (1.0, 0.0);

\fill[opacity=0.5, fill=white!80!blue] plot[smooth, tension=0.5] coordinates
  {(N0) (N1) (N2) (N3)};
\draw[color=blue] plot[smooth, tension=0.5] coordinates
  {(N0) (N1) (N2) (N3)};

\coordinate (B0) at (2.0, 0.0);
\coordinate (B1) at (1.7, 0.5);
\coordinate (B2) at (-1.0, 2.5);
\coordinate (B3) at (-0.5, 3.0);
\coordinate (B4) at (1.0, 2.5);
\coordinate (B5) at (3.0, 0.0);

\fill[opacity=0.5, fill=white!80!black] plot[smooth, tension=0.5] coordinates
  {(B0) (B1) (B2) (B3) (B4) (B5)};
\draw[color=black] plot[smooth, tension=0.5] coordinates
  {(B0) (B1) (B2) (B3) (B4) (B5)};
 
\coordinate (G0) at (3.3, 0.0);
\coordinate (G1) at (3.1, 0.7);
\coordinate (G2) at (3.6, 0.8);
\coordinate (G3) at (3.7, 0.0);

\draw[opacity=0.5, fill=white!80!gray] plot[smooth, tension=0.5] coordinates
  {(G0) (G1) (G2) (G3)};
\draw[color=gray] plot[smooth, tension=0.5] coordinates
  {(G0) (G1) (G2) (G3)};

\draw (-3,0) -- (6,0) {};
\end{tikzpicture}
\caption{A grounded family of sets}
\label{fig:grounded-family}
\end{figure}

All the geometric sets that we consider from now on are contained in $\setR\times[0,+\infty)$.
To prove Theorem \ref{thm:intro}, it suffices to show the following.

\begin{proposition}\label{prop:main}
For\/ $k\geq 1$, there is\/ $\xi_k$ such that\/ $\chi(\calF)\leq\xi_k$ holds for any grounded family\/ $\calF$ with\/ $\omega(\calF)\leq k$.
\end{proposition}

The case $k=1$ is trivial, and the case $k=2$ with some additional assumptions meant to avoid topological pathologies was settled by McGuinness \cite{McG00}.

We write $X\prec Y$ if $\base(X)$ is entirely to the left of $\base(Y)$.
The relation $\prec$ is a total order on a grounded family and naturally extends to its subfamilies (or any other families of grounded sets with pairwise disjoint bases): for example, $X\prec\calY$ denotes that $X\prec Y$ for any $Y\in\calY$.
For grounded sets $X_1$ and $X_2$ such that $X_1\prec X_2$, we define $\calF(X_1,X_2)=\{Y\in\calF\colon X_1\prec Y\prec X_2\}$.
For a grounded set $X$, we define $\calF(-\infty,X)=\{Y\in\calF\colon Y\prec X\}$ and $\calF(X,+\infty)=\{Y\in\calF\colon X\prec Y\}$.

The proof of Proposition \ref{prop:main} heavily depends on two decomposition lemmas, which given a grounded family with large chromatic number find its subfamily with large chromatic number and some special properties.
The first one is a reformulation of Lemma 2.1 in \cite{McG96}.

\begin{lemma}\label{lem:graph}
Let\/ $\calF$ be a grounded family with\/ $\chi(\calF)>2a(b+1)$, where\/ $a,b\geq 0$.
There is a subfamily\/ $\calH$ of\/ $\calF$ that satisfies\/ $\chi(\calH)>a$ and\/ $\chi(\calF(H_1,H_2))>b$ for any intersecting\/ $H_1,H_2\in\calH$.
\end{lemma}

\begin{proof}
We partition $\calF$ into subfamilies $\calF_0\prec\cdots\prec\calF_n$ so that $\chi(\calF_i)=b+1$ for $0\leq i<n$.
This can be done by adding sets to $\calF_0$ in the increasing $\prec$-order until we get $\chi(\calF_0)=b+1$, then following the same procedure with the remaining sets to form $\calF_1$, and so on.
Let $\calF^0=\bigcup_i\calF_{2i}$ and $\calF^1=\bigcup_i\calF_{2i+1}$.
Since $\chi(\calF^0\cup\calF^1)>2a(b+1)$, we have $\chi(\calF^k)>a(b+1)$ for $k=0$ or $k=1$.
We now color each $\calF_{2i+k}$ properly using the same set of $b+1$ colors.
This coloring induces a partitioning of the entire $\calF^k$ into subfamilies $\calH_0,\ldots,\calH_b$ such that for $0\leq i\leq n$, $0\leq j\leq b$ the family $\calF_i\cap\calH_j$ is independent.
We set $\calH=\calH_j$, where $\calH_j$ has the maximum chromatic number among $\calH_0,\ldots,\calH_b$.
Since $\chi(\calF^k)>a(b+1)$, we have $\chi(\calH)>a$.
It remains to show that $\chi(\calF(H_1,H_2))>b$ for $H_1,H_2\in\calH$ with $H_1\cap H_2\neq\emptyset$.
Indeed, such sets $H_1$ and $H_2$ must lie in different families $\calF_{2i_1+k}$ and $\calF_{2i_2+k}$, respectively, so $\chi(\calF(H_1,H_2))\geq\chi(\calF_{2i_1+k+1})=b+1>b$, as required.
\end{proof}

For a set $X$, we define $\ext(X)$ to be the only unbounded arc-connected component of $(\setR\times[0,+\infty))\setminus X$.
For a grounded family $\calF$, we define $\ext(\calF)=\ext(\bigcup\calF)$.
A subfamily $\calG$ of a grounded family $\calF$ is \emph{externally supported} in $\calF$ if for any $X\in\calG$ there exists $Y\in\calF$ such that $Y\cap X\neq\emptyset$ and $Y\cap\ext(\calG)\neq\emptyset$ (see Figure \ref{fig:externally-supported-family}).
The idea behind the following lemma is due to Gy\'{a}rf\'{a}s \cite{Gya85} and was subsequently used in \cite{McG96,McG00,Suk14}.

\begin{figure}[t]
\centering
\begin{tikzpicture}[xscale=.5]
\coordinate (A0) at (-2.0, 0.0);
\coordinate (A1) at (2.0, 3.0);
\coordinate (A2) at (3.0, 4.0);
\coordinate (A3) at (4.0, 3.0);
\coordinate (A4) at (4.0, 2.5);
\coordinate (A5) at (1.0, 1.5);
\coordinate (A6) at (-1.0, 0.0);

\fill[opacity=0.5, fill=white!90!black] plot[smooth, tension=0.5] coordinates
  {(A0) (A1) (A2) (A3) (A4) (A5) (A6)};
\draw[color=black] plot[smooth, tension=0.5] coordinates
  {(A0) (A1) (A2) (A3) (A4) (A5) (A6)};

\coordinate (B0) at (2.0, 0.0);
\coordinate (B1) at (2.0, 1.0);
\coordinate (B2) at (3.0, 1.0);
\coordinate (B3) at (3.0, 0.0);

\fill[opacity=0.5, fill=white!90!black] plot[smooth, tension=0.5] coordinates
  {(B0) (B1) (B2) (B3)};
\draw[color=black] plot[smooth, tension=0.5] coordinates
  {(B0) (B1) (B2) (B3)};

\coordinate (C0) at (6.0, 0.0);
\coordinate (C1) at (2.0, 3.0);
\coordinate (C2) at (1.8, 3.5);
\coordinate (C3) at (2.5, 3.5);
\coordinate (C4) at (7.0, 0.0);

\fill[opacity=0.5, fill=white!90!black] plot[smooth, tension=0.5] coordinates
  {(C0) (C1) (C2) (C3) (C4)};
\draw[color=black] plot[smooth, tension=0.5] coordinates
  {(C0) (C1) (C2) (C3) (C4)};

\coordinate (D0) at (4.0, 0.0);
\coordinate (D1) at (4.2, 1.0);
\coordinate (D2) at (4.8, 2.0);
\coordinate (D3) at (5.5, 1.5);
\coordinate (D4) at (5.0, 0.0);

\fill[opacity=0.5, fill=white!90!black] plot[smooth, tension=0.5] coordinates
  {(D0) (D1) (D2) (D3) (D4)};
\draw[color=black] plot[smooth, tension=0.5] coordinates
  {(D0) (D1) (D2) (D3) (D4)};

\coordinate (E0) at (8.0, 0.0);
\coordinate (E1) at (7.7, 0.5);
\coordinate (E15) at (7.9, 1.6);
\coordinate (E2) at (9.2, 1.9);
\coordinate (E3) at (9.0, 0.0);

\fill[opacity=0.5, fill=white!90!black] plot[smooth, tension=0.5] coordinates
  {(E0) (E1) (E15) (E2) (E3)};
\draw[color=black] plot[smooth, tension=0.5] coordinates
  {(E0) (E1) (E15) (E2) (E3)};

\coordinate (F0) at (12.0, 0.0);
\coordinate (F1) at (12.0, 1.0);
\coordinate (F2) at (13.0, 1.0);
\coordinate (F3) at (13.0, 0.0);

\fill[opacity=0.5, fill=white!90!black] plot[smooth, tension=0.5] coordinates
  {(F0) (F1) (F2) (F3)};
\draw[color=black] plot[smooth, tension=0.5] coordinates
  {(F0) (F1) (F2) (F3)};

\coordinate (G0) at (14.0, 0.0);
\coordinate (G1) at (12.0, 2.0);
\coordinate (G2) at (7.0, 3.0);
\coordinate (G3) at (4.0, 3.0);
\coordinate (G4) at (3.5, 3.5);
\coordinate (G5) at (8.0, 3.5);
\coordinate (G5) at (13.0, 3.0);
\coordinate (G6) at (15.0, 0.0);

\fill[opacity=0.5, fill=white!90!black] plot[smooth, tension=0.5] coordinates
  {(G0) (G1) (G2) (G3) (G4) (G5) (G6)};
\draw[color=black] plot[smooth, tension=0.5] coordinates
  {(G0) (G1) (G2) (G3) (G4) (G5) (G6)};

\coordinate (SA0) at (-4.0, 0.0);
\coordinate (SA1) at (-3.0, 1.5);
\coordinate (SA2) at (-1.0, 2.2);
\coordinate (SA3) at (3.0, 1.8);
\coordinate (SA4) at (4.2, 1.0);
\coordinate (SA5) at (4.1, 0.6);
\coordinate (SA6) at (3.1, 1.4);
\coordinate (SA7) at (2.0, 1.6);
\coordinate (SA8) at (-2.0, 1.0);
\coordinate (SA9) at (-3.0, 0.0);

\fill[opacity=0.5, fill=white!80!red] plot[smooth, tension=0.5] coordinates
  {(SA0) (SA1) (SA2) (SA3) (SA4) (SA5) (SA6) (SA7) (SA8) (SA9)};
\draw[dashed,color=red] plot[smooth, tension=0.5] coordinates
  {(SA0) (SA1) (SA2) (SA3) (SA4) (SA5) (SA6) (SA7) (SA8) (SA9)};

\coordinate (SB0) at (0.0, 0.0);
\coordinate (SB1) at (-2.0, 1.0);
\coordinate (SB2) at (-2.5, 2.0);
\coordinate (SB3) at (-2.0, 3.0);
\coordinate (SB4) at (0.0, 4.0);
\coordinate (SB5) at (2.0, 3.5);
\coordinate (SB6) at (1.6, 3.0);
\coordinate (SB7) at (0.9, 2.6);
\coordinate (SB8) at (1.8, 2.0);
\coordinate (SB9) at (2.2, 1.0);
\coordinate (SB11) at (2.0, 0.8);
\coordinate (SB13) at (0.8, 1.0);
\coordinate (SB14) at (-0.2, 3.0);
\coordinate (SB15) at (-1.0, 2.0);
\coordinate (SB16) at (1.0, 0.0);

\fill[opacity=0.5, fill=white!80!red] plot[smooth, tension=0.5] coordinates
  {(SB0) (SB1) (SB2) (SB3) (SB4) (SB5) (SB6) (SB7) (SB8) (SB9) (SB11) (SB13) (SB14) (SB15) (SB16)};
\draw[dashed,color=red] plot[smooth, tension=0.5] coordinates
  {(SB0) (SB1) (SB2) (SB3) (SB4) (SB5) (SB6) (SB7) (SB8) (SB9) (SB11) (SB13) (SB14) (SB15) (SB16)};

\coordinate (SC0) at (10.0, 0.0);
\coordinate (SC1) at (9.0, 1.0);
\coordinate (SC2) at (8.0, 1.5);
\coordinate (SC3) at (7.5, 2.0);
\coordinate (SC4) at (7.0, 2.5);
\coordinate (SC5) at (6.5, 3.5);
\coordinate (SC6) at (7.0,4.0);
\coordinate (SC7) at (7.5, 3.9);
\coordinate (SC8) at (8.0, 3.5);
\coordinate (SC9) at (9.0, 2.5);
\coordinate (SC10) at (10.0, 2.0);
\coordinate (SC11) at (11.0, 0.0);

\fill[opacity=0.5, fill=white!80!red] plot[smooth, tension=0.5] coordinates
  {(SC0) (SC1) (SC2) (SC3) (SC4) (SC5) (SC6) (SC7) (SC8) (SC10) (SC11)};
\draw[dashed,color=red] plot[smooth, tension=0.5] coordinates
  {(SC0) (SC1) (SC2) (SC3) (SC4) (SC5) (SC6) (SC7) (SC8) (SC10) (SC11)};
  
\coordinate (SD0) at (16.0, 0.0);
\coordinate (SD1) at (14.0, 1.0);
\coordinate (SD2) at (13.0, 1.3);
\coordinate (SD3) at (11.8, 0.5);
\coordinate (SD4) at (11.6, 1.0);
\coordinate (SD5) at (12.0, 1.8);
\coordinate (SD6) at (13.0, 2.1);
\coordinate (SD7) at (14.5, 1.5);
\coordinate (SD11) at (17.0, 0.0);

\fill[opacity=0.5, fill=white!80!red] plot[smooth, tension=0.5] coordinates
  {(SD0) (SD1) (SD2) (SD3) (SD4) (SD5) (SD6) (SD7) (SD11)};
\draw[dashed,color=red] plot[smooth, tension=0.5] coordinates
  {(SD0) (SD1) (SD2) (SD3) (SD4) (SD5) (SD6) (SD7) (SD11)};
  
\draw (-5,0) -- (18,0) {};
\end{tikzpicture}
\caption{An externally supported family of sets}
\label{fig:externally-supported-family}
\end{figure}
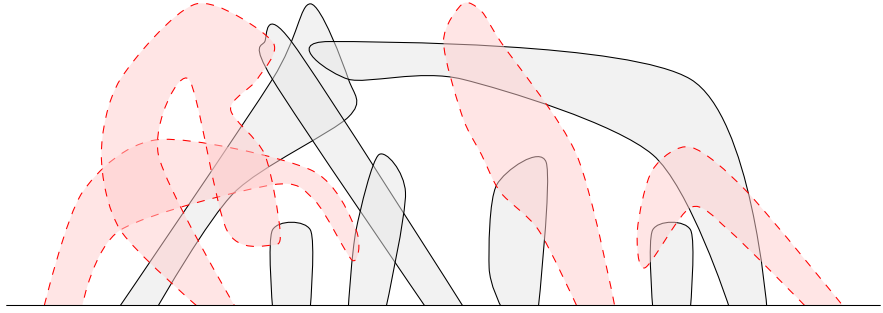

\begin{lemma}\label{lem:level}
Let\/ $\calF$ be a grounded family with\/ $\chi(\calF)>2a$, where\/ $a\geq 1$.
There is a subfamily\/ $\calG$ of\/ $\calF$ that is externally supported in\/ $\calF$ and satisfies\/ $\chi(\calG)>a$.
\end{lemma}

\begin{proof}
For convenience, we restrict $\calF$ to its connected component with maximum chromatic number.
Let $X_0$ be the $\prec$-least member of $\calF$.
For $i\geq 0$, let $\calF_i$ be the family of members of $\calF$ that are at distance $i$ from $X_0$ in the intersection graph of $\calF$.
It follows that $\calF_0=\{X_0\}$ and, for $|i-j|>1$, each member of $\calF_i$ is disjoint from each member of $\calF_j$.
Clearly, $\chi(\bigcup_i\calF_{2i})>a$ or $\chi(\bigcup_i\calF_{2i+1})>a$, and therefore there is $d\geq 1$ with $\chi(\calF_d)>a$.
We claim that $\calF_d$ is externally supported in $\calF$. 
Fix $X_d\in\calF_d$, and let $X_0,\ldots,X_d$ be a shortest path from $X_0$ to $X_d$ in the intersection graph of $\calF$.
Since $X_0\cap\ext(\calF_d)\neq\emptyset$ and $X_0,\ldots,X_{d-2}$ are disjoint from $\bigcup\calF_d$, we have $X_0,\ldots,X_{d-2}\subset\ext(\calF_d)$.
Thus $X_{d-1}\cap\ext(\calF_d)\neq\emptyset$ and $X_{d-1}\cap X_d\neq\emptyset$.
\end{proof}

\section{Cliques and brackets}

Let $\calF$ be a grounded family with $\omega(\calF)\leq k$.
A \emph{$k$-clique} in $\calF$ is a family of $k$ pairwise intersecting members of $\calF$.
For a $k$-clique $\calK$, we denote by $\int(\calK)$ the only arc-connected component of $(\setR\times[0,+\infty))\setminus\bigcup\calK$ containing the part of the baseline between the two $\prec$-least members of $\calK$.
A \emph{$k$-bracket} in $\calF$ is a subfamily of $\calF$ consisting of a $k$-clique $\calK$ and a set $S$ called the \emph{support} such that $S\prec\calK$ or $\calK\prec S$ and $S\cap\int(\calK)\neq\emptyset$.
For such a $k$-bracket $\calB$, we denote by $\int(\calB)$ the only arc-connected component of $(\setR\times[0,+\infty))\setminus\bigcup\calB$ containing the part of the baseline between $S$ and $\calK$ (see Figure \ref{fig:k-bracket}).

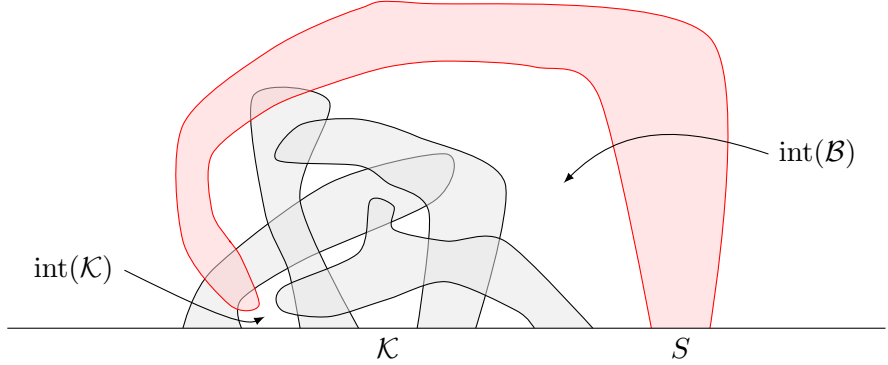
\begin{figure}[t]
\centering
\begin{tikzpicture}[scale=.77,>=latex]
\coordinate (A0) at (0.0, 0.0);
\coordinate (A1) at (0.5, 1.0);
\coordinate (A2) at (2.5, 2.5);
\coordinate (A3) at (4.5, 3.0);
\coordinate (A4) at (4.3, 2.2);
\coordinate (A5) at (2.1, 1.2);
\coordinate (A6) at (1.0, 0.5);
\coordinate (A7) at (1.0, 0.0);

\fill[opacity=0.5, fill=white!90!black] plot[smooth, tension=0.5] coordinates
  {(A0) (A1) (A2) (A3) (A4) (A5) (A6) (A7)};
\draw[color=black] plot[smooth, tension=0.5] coordinates
  {(A0) (A1) (A2) (A3) (A4) (A5) (A6) (A7)};

\coordinate (B0) at (2.0, 0.0);
\coordinate (B1) at (1.8, 1.0);
\coordinate (B2) at (1.4, 2.0);
\coordinate (B3) at (1.2, 4.0);
\coordinate (B4) at (2.5, 3.9);
\coordinate (B5) at (2.0, 2.2);
\coordinate (B6) at (3.0, 0.0);

\fill[opacity=0.5, fill=white!90!black] plot[smooth, tension=0.5] coordinates
  {(B0) (B1) (B2) (B3) (B4) (B5) (B6)};
\draw[color=black] plot[smooth, tension=0.5] coordinates
  {(B0) (B1) (B2) (B3) (B4) (B5) (B6)};

\coordinate (C0) at (4.0, 0.0);
\coordinate (C1) at (4.2, 2.0);
\coordinate (C2) at (3.5, 2.5);
\coordinate (C3) at (3.0, 2.8);
\coordinate (C4) at (1.6, 2.9);
\coordinate (C5) at (2.0, 3.5);
\coordinate (C6) at (3.0, 3.6);
\coordinate (C7) at (4.0, 3.3);
\coordinate (C8) at (5.5, 2.5);
\coordinate (C9) at (5.0, 0.0);

\fill[opacity=0.5, fill=white!90!black] plot[smooth, tension=0.5] coordinates
  {(C0) (C1) (C2) (C3) (C4) (C5) (C6) (C7) (C8) (C9)};
\draw[color=black] plot[smooth, tension=0.5] coordinates
  {(C0) (C1) (C2) (C3) (C4) (C5) (C6) (C7) (C8) (C9)};

\coordinate (D0) at (6.0, 0.0);
\coordinate (D1) at (5.7, 0.3);
\coordinate (D2) at (4.5, 0.8);
\coordinate (D3) at (3.5, 0.5);
\coordinate (D4) at (2.5, 0.2);
\coordinate (D5) at (1.7, 0.3);
\coordinate (D6) at (1.7, 0.7);
\coordinate (D7) at (3.0, 1.3);
\coordinate (D8) at (3.2, 2.2);
\coordinate (D9) at (3.6, 2.1);
\coordinate (D10) at (3.6, 1.7);
\coordinate (D11) at (4.5, 1.5);
\coordinate (D12) at (5.5, 1.5);
\coordinate (D13) at (7.0, 0.0);

\fill[opacity=0.5, fill=white!90!black] plot[smooth, tension=0.5] coordinates
  {(D0) (D1) (D2) (D3) (D4) (D5) (D6) (D7) (D8) (D9) (D10) (D11) (D12) (D13)};
\draw[color=black] plot[smooth, tension=0.5] coordinates
  {(D0) (D1) (D2) (D3) (D4) (D5) (D6) (D7) (D8) (D9) (D10) (D11) (D12) (D13)};

\coordinate (E0) at (8.0, 0.0);
\coordinate (E1) at (7.1, 4.0);
\coordinate (E2) at (6.0, 4.5);
\coordinate (E3) at (4.5, 4.6);
\coordinate (E4) at (3.5, 4.5);
\coordinate (E5) at (2.5, 4.2);
\coordinate (E6) at (1.5, 3.8);
\coordinate (E7) at (0.5, 3.0);
\coordinate (E8) at (0.5, 2.0);
\coordinate (E9) at (0.8, 1.5);
\coordinate (E10) at (0.8, 1.5);
\coordinate (E11) at (1.0, 1.2);
\coordinate (E12) at (1.3, 0.4);
\coordinate (E13) at (0.8, 0.4);
\coordinate (E14) at (0.0, 1.5);
\coordinate (E15) at (0.0, 3.5);
\coordinate (E16) at (1.5, 4.8);
\coordinate (E17) at (3.0, 5.5);
\coordinate (E18) at (4.0, 5.6);
\coordinate (E19) at (9.0, 5.0);
\coordinate (E20) at (9.0, 0.0);

\fill[opacity=0.5, fill=white!80!red] plot[smooth, tension=0.5] coordinates
  {(E0) (E1) (E2) (E3) (E4) (E5) (E6) (E7) (E8) (E9) (E10) (E11) (E12) (E13) (E14) (E15) (E16) (E17) (E18) (E19) (E20)};
\draw[color=red] plot[smooth, tension=0.5] coordinates
  {(E0) (E1) (E2) (E3) (E4) (E5) (E6) (E7) (E8) (E9) (E10) (E11) (E12) (E13) (E14) (E15) (E16) (E17) (E18) (E19) (E20)};

\draw (3.5,0) node[below] {$\calK$};
\draw (8.5,0) node[below] {$S$};
\draw[->] (10,3) node[below,right] {$\int(\calB)$} .. controls (8.5,3.5) and (7.5,3.5) .. (6.5,2.5);
\draw[->] (-1,1) node[below,left] {$\int(\calK)$} .. controls (0,0.5) and (1,0) .. (1.4,0.2);

\draw (-3,0) -- (12,0) {};
\end{tikzpicture}
\caption{A $k$-bracket $\calB$ with $k$-clique $\calK$ and support $S$.}
\label{fig:k-bracket}
\end{figure}

\begin{lemma}\label{lem:surround}
Let\/ $\calF$ be a grounded family, $X,Y,Z\in\calF$, and\/ $X\cap Y\neq\emptyset$.
Let\/ $C_1$ and\/ $C_2$ be any two distinct arc-connected components of\/ $(\setR\times[0,+\infty))\setminus(X\cup Y)$.
Let\/ $z,z'\in Z\cap C_1$.
If every arc between\/ $z$ and\/ $z'$ within\/ $Z$ intersects\/ $C_2$, then every such arc intersects both\/ $X$ and\/ $Y$.
\end{lemma}

\begin{proof}
Suppose there is an arc $A\subset Z$ between $z$ and $z'$ such that $A\cap X=\emptyset$ or $A\cap Y=\emptyset$.
If $A\cap X=\emptyset$ and $A\cap Y=\emptyset$, then $A\subset C_1$, so $A\cap C_2=\emptyset$.
Now, suppose $A\cap X=\emptyset$ and $A\cap Y\neq\emptyset$.
Let $y$ and $y'$ be respectively the first and last points of $A$ in $Y$.
Since $Z\cap Y$ is arc-connected, there is an arc $A'$ in $Z$ that is simple with respect to $Y$ and goes along $A$ from $z$ to $y$, then to $y'$ inside $Y$, and finally along $A$ to $z'$.
It follows that $A'\subset C_1\cup Y$, so $A'\cap C_2=\emptyset$.
The case that $A\cap X\neq\emptyset$ and $A\cap Y=\emptyset$ is symmetric.
\end{proof}

\begin{corollary}\label{cor:clique}
Let\/ $\calF$ be a grounded family, $\calK$ be a\/ $k$-clique in\/ $\calF$, and\/ $X\in\calF$.
If\/ $x,y\in X\cap\int(\calK)$ (or\/ $x,y\in X\cap\ext(\calK)$) and every arc between\/ $x$ and\/ $y$ within\/ $X$ intersects\/ $\ext(\calK)$ (\/$\int(\calK)$, respectively), then\/ $X$ intersects every member of\/ $\calK$.
\end{corollary}

\begin{proof}
Let $\calK=\{K_1,\ldots,K_k\}$ and $K_1\prec\int(\calK)\prec K_2\prec\cdots\prec K_k$.
The statement follows directly from Lemma \ref{lem:surround} and the fact that $\int(\calK)$ and $\ext(\calK)$ belong to distinct arc-connected components of $(\setR\times[0,+\infty))\setminus(K_1\cup K_i)$ for $2\leq i\leq k$.
\end{proof}

\begin{corollary}\label{cor:bracket}
Let\/ $\calF$ be a grounded family and\/ $\calB$ be a bracket in\/ $\calF$ with clique\/ $\calK$ and support\/ $S$.
Let\/ $X\in\calF$.
If\/ $X\cap\int(\calB)\neq\emptyset$ and\/ $X\cap\ext(\calB)\neq\emptyset$, then\/ $X$ intersects\/ $S$ or every member of\/ $\calK$.
\end{corollary}

\begin{proof}
Let $x\in X\cap\int(\calB)$ and $x'\in X\cap\ext(\calB)$, and suppose $X\cap S=\emptyset$.
By the Jordan curve theorem, every arc between $x$ and $x'$ within $X$ must intersect $S\cup\int(\calK)$ and thus $\int(\calK)$.
Since $x,x'\in X\cap\ext(\calK)$, it follows from Corollary \ref{cor:clique} that $X$ intersects every member of $\calK$.
\end{proof}

\section{Proof of Proposition \ref{prop:main}}

The proof goes by induction on $k$.
Proposition \ref{prop:main} holds trivially for $k=1$ with $\xi_1=1$.
Therefore, we assume that $k\geq 2$ and that the statement of the proposition holds for $k-1$.
This context of the induction step is maintained throughout the entire remaining part of the paper.
A typical application of the induction hypothesis looks as follows: if $\calF$ is a grounded family with $\omega(\calF)\leq k$, $\calG\subset\calF$, and there is $X\in\calF\setminus\calG$ intersecting all members of $\calG$, then $\omega(\calG)\leq k-1$ and thus $\chi(\calG)\leq\xi_{k-1}$.

Define $\beta_k=8k\xi_{k-1}^2$, $\delta_{k,k}=0$, $\delta_{k,j}=\beta_k+2\delta_{k,j+1}+2\xi_{k-1}(k\xi_{k-1}+k+2)+2$ for $k-1\geq j\geq 0$, and finally $\xi_k=2^{k+2}(\delta_{k,0}+2\xi_{k-1}+1)$.

We say that a grounded set $X$ (a grounded family $\calX$) is \emph{surrounded} by a set $S$ if $X$ (every member of $\calX$, respectively) is disjoint from $S\cup\ext(S)$.
For a set $S$ and a grounded set $R$ such that $\base(R)$ is surrounded by $S$, let $\cut(R,S)$ denote the closure of the unique arc-connected component of $R\setminus S$ containing $\base(R)$.
For a set $S$ and a grounded family $\calR$ of sets whose bases are surrounded by $S$, let $\cut(\calR,S)=\{\cut(R,S)\colon R\in\calR\}$.

First, we present a technical lemma, which generalizes similar statements from \cite{McG00} (Lemma 3.2) and \cite{Suk14} (Lemma 4.1), and which we will prove in Section \ref{sec:lemma}.
Loosely speaking, it says that one can color properly, with the number of colors bounded in terms of $k$, all the members of $\calF$ surrounded by a set $S$ which intersect $\cut(R,S)$ for any set $R\in\calF$ intersecting $S$.

\begin{lemma}\label{lem:dist2}
Let\/ $S$ be a compact set and\/ $\calR\cup\calD$ be a grounded family with the following properties:
\begin{itemize}
\item the base of every member of\/ $\calR$ is surrounded by\/ $S$,
\item every member of\/ $\calR$ intersects\/ $S$,
\item $\calD$ is surrounded by\/ $S$,
\item every member of\/ $\calD$ intersects\/ $\bigcup\cut(\calR,S)$,
\item $\omega(\calR\cup\calD)\leq k$.
\end{itemize}
It follows that\/ $\chi(\calD)\leq\beta_k$.
\end{lemma}

Suppose for the sake of contradiction that there is a grounded family $\calF$ with $\omega(\calF)\leq k$ and $\chi(\calF)>\xi_k=2^{k+2}(\delta_{k,0}+2\xi_{k-1}+1)$.
A repeated application of Lemma \ref{lem:level} yields a sequence of families $\calF=\calF_{k+1}\supset\calF_k\supset\cdots\supset\calF_0$ such that $\calF_i$ is externally supported in $\calF_{i+1}$ and $\chi(\calF_i)>2^{i+1}(\delta_{k,0}+2\xi_{k-1}+1)$, for $0\leq i\leq k$.
The following claim is the core of the proof.

\begin{claim}\label{cla:main}
For\/ $0\leq j\leq k$, there are families\/ $\calS,\calG\subset\calF_j$ and sets\/ $S_1,\ldots,S_j\in\calS$ with the following properties:
\begin{enumerate}
\item $\calG$ is surrounded by\/ $\bigcup\calS$,
\item $\chi(\calG)>\delta_{k,j}$,
\item the sets\/ $S_1,\ldots,S_j$ pairwise intersect,
\item every member of\/ $\calF$ intersecting\/ $\ext(\calS)$ and some member of\/ $\calG$ also intersects each of\/ $S_1,\ldots,S_j$.
\end{enumerate}
\end{claim}

\begin{proof}
The proof goes by induction on $j$.
First, let $j=0$.
Apply Lemma \ref{lem:graph} to find $\calH\subset\calF_0$ such that $\chi(\calH)>1$ and for any intersecting $H_1,H_2\in\calH$ we have $\chi(\calF_0(H_1,H_2))>\delta_{k,0}+2\xi_{k-1}$.
Since $\chi(\calH)>1$, such two intersecting $H_1,H_2\in\calH$ exist.
Let $\calS=\{H_1,H_2\}$ and $\calG$ be the family of those members of $\calF_0(H_1,H_2)$ that are disjoint from $H_1\cup H_2$.
It is clear that (i) holds.
Since the members of $\calF_0(H_1,H_2)$ intersecting $H_1\cup H_2$ have chromatic number at most $2\xi_{k-1}$, we have $\chi(\calG)>\delta_{k,0}$, so (ii) holds.
The conditions (iii) and (iv) are satisfied vacuously.

Now, assume that $j\geq 1$ and the claim holds for $j-1$, that is, there are families $\calS',\calG'\subset\calF_{j-1}$ and sets $S_1,\ldots,S_{j-1}\in\calS'$ satisfying (i)--(iv).
Let
\begin{align*}
\calR&=\bigl\{R\in\calF\colon\base(R)\text{ is surrounded by }\tbigcup\calS'\text{ and }R\cap\tbigcup\calS'\neq\emptyset\bigr\},\\
\calD&=\bigl\{D\in\calG'\colon D\cap\tbigcup\cut(\calR,\tbigcup\calS')\neq\emptyset\bigr\}.
\end{align*}
It follows from Lemma \ref{lem:dist2} that $\chi(\calD)\leq\beta_k$ and thus $\chi(\calG'\setminus\calD)>2\delta_{k,j}+2\xi_{k-1}(k\xi_{k-1}+k+2)+2$.
Since the chromatic number of a graph is the maximum chromatic number of its connected component, there is $\calG''\subset\calG'\setminus\calD$ such that the intersection graph of $\calG''$ is connected and $\chi(\calG'')>2\delta_{k,j}+2\xi_{k-1}(k\xi_{k-1}+k+2)+2$.
Partition $\calG''$ into three subfamilies $\calX,\calY,\calZ$ so that $\calX\prec\calY\prec\calZ$ and $\chi(\calX)=\chi(\calZ)=\delta_{k,j}+(k+1)\xi_{k-1}+1$.
It follows that $\chi(\calY)>2\xi_{k-1}(k\xi_{k-1}+1)$.
Apply Lemma \ref{lem:graph} to find $\calH\subset\calY$ such that $\chi(\calH)>\xi_{k-1}$ and for any intersecting $H_1,H_2\in\calH$ we have $\chi(\calY(H_1,H_2))>k\xi_{k-1}$.
Since $\chi(\calH)>\xi_{k-1}$, there is a $k$-clique $\calK\subset\calH$.
The members of $\calY$ intersecting $\bigcup\calK$ have chromatic number at most $k\xi_{k-1}$, so there is $P\in\calY$ that is contained in $\int(\calK)$.
Since $\calF_{j-1}$ is externally supported in $\calF_j$, there is $S_j\in\calF_j$ such that $S_j\cap P\neq\emptyset$ and $S_j\cap\ext(\calS')\supset S_j\cap\ext(\calF_{j-1})\neq\emptyset$.
Therefore, since $\calS'$ and $\calG'$ satisfy (iv), $S_j$ intersects each of $S_1,\ldots,S_{j-1}$ and thus (iii) holds for $S_1,\ldots,S_j$.

We show that $S_j\prec\calG''$ or $\calG''\prec S_j$.
Suppose that neither of these holds.
It follows that $\base(S_j)$ is surrounded by $\bigcup\calS'$, which yields $S_j\in\calR$.
Moreover, $\base(S_j)$ is surrounded by $\bigcup\calG''$, as the intersection graph of $\calG''$ is connected.
Therefore, we have $\cut(S_j,\bigcup\calS')\cap\bigcup\calG''\neq\emptyset$, so there is $X\in\calG''$ such that $X\cap\cut(S_j,\bigcup\calS')\neq\emptyset$.
This means that $X\in\calD$, which contradicts the definition of $\calG''$.

Now, we have $S_j\cap\int(\calK)\supset S_j\cap P\neq\emptyset$ and $S_j\prec\calK$ or $\calK\prec S_j$, so the $k$-clique $\calK$ and the support $S_j$ form a $k$-bracket.
Let $\calS=\calS'\cup\calK\cup\{S_j\}$.
If $S_j\prec\calG''$, then $S_j\prec\calX\prec\calK$.
In this case, let $\calG$ be the family of those members of $\calX$ that are disjoint from $\bigcup\calK\cup S_j$.
It is clear that (i) holds.
Since $\chi(\calX)>\delta_{k,j}+(k+1)\xi_{k-1}$ and the members of $\calX$ intersecting $\bigcup\calK\cup S_j$ have chromatic number at most $(k+1)\xi_{k-1}$, we have $\chi(\calG)>\delta_{k,j}$, so (ii) holds.
Since $\ext(\calS)\subset\ext(\calK\cup\{S_j\})$, it follows from Corollary \ref{cor:bracket} that every member of $\calF$ intersecting $\ext(\calS)$ and some member of $\calG$ intersects $S_j$.
Hence (iv) holds.
If $\calG''\prec S_j$, then let $\calG$ be the family of those members of $\calZ$ that are disjoint from $\bigcup\calK\cup S_j$.
An analogous argument shows that (i), (ii), and (iv) are satisfied.
\end{proof}

Let $\calS$, $\calG$, and $S_1,\ldots,S_k$ be as guaranteed by Claim \ref{cla:main} for $j=k$.
By (ii), we have $\chi(\calG)>0$, so there is $P\in\calG$.
Since $\calF_k$ is externally supported in $\calF$, there is $S_{k+1}\in\calF$ such that $S_{k+1}\cap P\neq\emptyset$ and $S_{k+1}\cap\ext(\calS)\neq\emptyset$.
By (iii) and (iv), we conclude that $S_1,\ldots,S_{k+1}$ pairwise intersect.
This contradicts the assumption that $\omega(\calF)\leq k$, thus completing the proof of Proposition \ref{prop:main}.

\section{Proof of Lemma \ref{lem:dist2}}\label{sec:lemma}

The proof of Lemma \ref{lem:dist2} goes along similar lines to the proof of Lemma 4.1 in \cite{Suk14}.
Since $\calD$ is surrounded by $S$, there is an arc $S'\subset S$ such that $\calD$ is surrounded by $S'$.
We can assume without loss of generality that $\base(S')=\{p,q\}$ for some points $p$ and $q$ on the baseline such that $p\prec\calD\prec q$.
For every $R\in\calR$ we have $\cut(R,S)\subset\cut(R,S')$.
Hence every member of $\calD$ intersects $\bigcup\cut(R,S')$ (see Figure \ref{fig:cutR}).

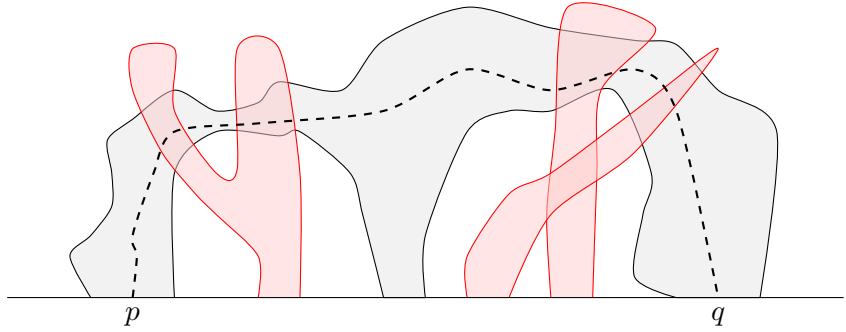
\begin{figure}[t]
\centering
\begin{tikzpicture}[scale=.55]
\coordinate (A0) at (0.0, 0.0);
\coordinate (A1) at (-0.5, 1.0);
\coordinate (A2) at (0.0, 1.5);
\coordinate (A3) at (0.5, 2.2);
\coordinate (A4) at (0.5, 3.0);
\coordinate (A5) at (0.4, 3.8);
\coordinate (A6) at (1.0, 4.3);
\coordinate (A7) at (2.0, 5.0);
\coordinate (A8) at (3.0, 4.5);
\coordinate (A9) at (4.0, 4.7);
\coordinate (A10) at (4.5, 5.2);
\coordinate (A11) at (6.0, 5.0);
\coordinate (A12) at (7.0, 6.2);
\coordinate (A13) at (9.0, 7.0);
\coordinate (A14) at (11.0, 6.5);
\coordinate (A15) at (13.0, 6.2);
\coordinate (A16) at (14.0, 6.1);
\coordinate (A17) at (15.0, 5.0);
\coordinate (A18) at (16.4, 4.0);
\coordinate (A19) at (16.0, 0.0);

\coordinate (B0) at (14.0, 0.0);
\coordinate (B1) at (13.0, 0.5);
\coordinate (B2) at (13.2, 2.0);
\coordinate (B3) at (13.4, 3.0);
\coordinate (B4) at (12.5, 5.0);
\coordinate (B5) at (11.0, 4.5);
\coordinate (B6) at (10.0, 4.5);
\coordinate (B7) at (9.0, 4.0);
\coordinate (B8) at (8.0, 1.5);
\coordinate (B9) at (8.0, 0.0);

\coordinate (C0) at (7.0, 0.0);
\coordinate (C1) at (6.5, 2.0);
\coordinate (C2) at (6.2, 3.0);
\coordinate (C3) at (5.0, 4.0);
\coordinate (C4) at (4.5, 3.9);
\coordinate (C5) at (3.0, 4.0);

\coordinate (Z4) at (2,3);
\coordinate (Z5) at (2,0);

\fill[opacity=0.5, fill=white!90!black] plot[smooth, tension=0.5] coordinates
  {(A0) (A1) (A2) (A3) (A4) (A5) (A6) (A7) (A8) (A9) (A10) (A11) (A12) (A13) (A14) (A15) (A16) (A17) (A18) (A19)}
  -- (B0) plot[smooth,tension=0.5] coordinates
  {(B0) (B1) (B2) (B3) (B4) (B5) (B6) (B7) (B8) (B9)}
  -- (C0) plot[smooth,tension=0.5] coordinates
  {(C0) (C1) (C2) (C3) (C4) (C5) (Z4) (Z5)};
\draw[color=black] plot[smooth, tension=0.5] coordinates
  {(A0) (A1) (A2) (A3) (A4) (A5) (A6) (A7) (A8) (A9) (A10) (A11) (A12) (A13) (A14) (A15) (A16) (A17) (A18) (A19)}
  -- (B0) plot[smooth,tension=0.5] coordinates
  {(B0) (B1) (B2) (B3) (B4) (B5) (B6) (B7) (B8) (B9)}
  -- (C0) plot[smooth,tension=0.5] coordinates
  {(C0) (C1) (C2) (C3) (C4) (C5) (Z4) (Z5)};

\coordinate (S0) at (1.0, 0.0);
\coordinate (S1) at (1.1, 1.0);
\coordinate (S2) at (1.0, 1.5);
\coordinate (S3) at (1.5, 3.0);
\coordinate (S4) at (2.0, 4.0);
\coordinate (S5) at (4.0, 4.2);
\coordinate (S6) at (7.0, 4.5);
\coordinate (S7) at (9.0, 5.5);
\coordinate (S8) at (11.0, 5.0);
\coordinate (S9) at (13.0, 5.5);
\coordinate (S10) at (14.0, 4.5);
\coordinate (S11) at (15.0, 0.0);

\draw (S0) node[below] {$p$};
\draw (S11) node[below] {$q$};

\coordinate (R0) at (4.0, 0.0);
\coordinate (R1) at (4.0, 1.0);
\coordinate (R2) at (3.0, 2.0);
\coordinate (R3) at (2.5, 2.5);
\coordinate (R4) at (1.7, 3.5);
\coordinate (R5) at (1.0, 5.0);
\coordinate (R6) at (1.0, 6.0);
\coordinate (R7) at (2.0, 6.0);
\coordinate (R8) at (2.0, 4.5);
\coordinate (R9) at (3.0, 3.0);
\coordinate (R10) at (3.5, 3.1);
\coordinate (R11) at (3.5, 6.0);
\coordinate (R12) at (4.5, 6.0);
\coordinate (R13) at (5.0, 3.0);
\coordinate (R14) at (5.0, 0.0);

\fill[opacity=0.5, fill=white!80!red] plot[smooth, tension=0.5] coordinates
  {(R0) (R1) (R2) (R3) (R4) (R5) (R6) (R7) (R8) (R9) (R10) (R11) (R12) (R13) (R14)};
\draw[color=red] plot[smooth, tension=0.5] coordinates
  {(R0) (R1) (R2) (R3) (R4) (R5) (R6) (R7) (R8) (R9) (R10) (R11) (R12) (R13) (R14)};

\coordinate (T0) at (11.0, 0.0);
\coordinate (T1) at (10.9, 2.0);
\coordinate (T2) at (11.1, 4.0);
\coordinate (T3) at (11.4, 7.0);
\coordinate (T4) at (13.5, 6.5);
\coordinate (T5) at (12.2, 5.0);
\coordinate (T6) at (12.1, 3.0);
\coordinate (T7) at (12.0, 0.0);

\fill[opacity=0.5, fill=white!80!red] plot[smooth, tension=0.5] coordinates
  {(T0) (T1) (T2) (T3) (T4) (T5) (T6) (T7)};
\draw[color=red] plot[smooth, tension=0.5] coordinates
  {(T0) (T1) (T2) (T3) (T4) (T5) (T6) (T7)};

\coordinate (U0) at (9.0, 0.0);
\coordinate (U1) at (9.0, 1.0);
\coordinate (U2) at (10.0, 2.5);
\coordinate (U3) at (11.0, 3.0);
\coordinate (U4) at (13.0, 4.5);
\coordinate (U5) at (15.0, 6.0);
\coordinate (U6) at (16.0, 6.0);
\coordinate (U6) at (13.0, 3.5);
\coordinate (U7) at (11.0, 2.0);
\coordinate (U8) at (10.0, 0.0);

\fill[opacity=0.5, fill=white!80!red] plot[smooth, tension=0.5] coordinates
  {(U0) (U1) (U2) (U3) (U4) (U5) (U6) (U7) (U8)};
\draw[color=red] plot[smooth, tension=0.5] coordinates
  {(U0) (U1) (U2) (U3) (U4) (U5) (U6) (U7) (U8)};

\path[draw, thick, dashed] plot[smooth, tension=0.5] coordinates
  {(S0) (S1) (S2) (S3) (S4) (S5) (S6) (S7) (S8) (S9) (S10) (S11)};

\draw (-2,0) -- (18,0) {};
\end{tikzpicture}
\caption{The setting of the proof of Lemma \ref{lem:dist2}: the set $S$ with a dashed arc $S'$ and three sets from $\calR$.}
\label{fig:cutR}
\end{figure}

\begin{claim}\label{cla:pillars}
$\chi(\cut(\calR,S'))\leq k$.
\end{claim}

\begin{proof}
Consider a relation $<$ on $\cut(\calR,S')$ defined as follows: $R_1'<R_2'$ if and only if $R_1'\prec R_2'$ and $R_1'\cap R_2'=\emptyset$.
It is clear that $<$ is irreflexive and antisymmetric.
It is also transitive, which follows from the fact that if $R_1',R_2',R_3'\in\calR$, $R_1'\prec R_2'\prec R_3'$, and $R_1'\cap R_3'\neq\emptyset$, then $R_2'\cap(R_1'\cup R_3')\neq\emptyset$.
Therefore, $<$ is a strict partial order.
The intersection graph of $\cut(\calR,S')$ is the incomparability graph of $<$, so it is perfect, which implies $\chi(\cut(\calR,S'))=\omega(\cut(\calR,S'))\leq k$.
\end{proof}

By Claim \ref{cla:pillars}, there is a coloring $\phi$ of $\calR$ with $k$ colors such that for any $R_1,R_2\in\calR$ with $\phi(R_1)=\phi(R_2)$, we have $\cut(R_1,S')\cap\cut(R_2,S')=\emptyset$.
For a color $c$, let $\calR^c=\{R\in\calR\colon\phi(R)=c\}$ and $\calD^c=\{D\in\calD\colon D\cap\bigcup\cut(\calR^c,S')\neq\emptyset\}$.
We are going to show that $\chi(\calD^c)\leq 8\xi_{k-1}^2$.
Once this is obtained, we will have $\chi(\calD)\leq\sum_c\chi(\calD^c)\leq 8k\xi_{k-1}^2=\beta_k$.

Since the sets $\cut(R,S')$ for $R\in\calR^c$ are pairwise disjoint, the curve $S'$ and the families $\cut(\calR^c,S')$ and $\calD^c$ satisfy the assumptions of Lemma \ref{lem:dist2}.
To simplify the notation, we assume for the remainder of the proof that $S=S'$, $R=\cut(R,S')$ for every $R\in\calR^c$, $\calR=\cut(\calR^c,S')$, and $\calD=\calD^c$.
By Jordan-Sch\"{o}nflies theorem, the segment $pq$ and the arc $S$ form a Jordan curve which is the boundary of a set $J$ homeomorphic to a closed disc.
In this new setting, $S$ is an arc and $\calR\cup\calD$ is a grounded family with the following properties:
\begin{itemize}
\item the base of every member of $\calR$ is surrounded by $S$,
\item every member of $\calR$ is contained in $J$ and intersects $S$,
\item the members of $\calR$ are pairwise disjoint,
\item $\calD$ is surrounded by $S$,
\item every member of $\calD$ intersects $\bigcup\calR$,
\item $\omega(\calR\cup\calD)\leq k$.
\end{itemize}
We enumerate the members of $\calR$ as $R_1,\ldots,R_m$ in the $\prec$-order, that is, so that $R_1\prec\cdots\prec R_m$.
We are going to show that $\chi(\calD)\leq 8\xi_{k-1}^2$.

\begin{claim}\label{cla:inters}
For\/ $1\leq i<j\leq m$, any arc in\/ $J$ between\/ $R_i$ and\/ $R_j$ intersects all\/ $R_{i+1},\ldots,R_{j-1}$.
\end{claim}

\begin{proof}
Let $A$ be an arc in $J$ between points $x_i\in R_i$ and $x_j\in R_j$.
For any $R\in\{R_{i+1},\ldots,R_{j-1}\}$, $\base(R)$ is surrounded by $R_i\cup A\cup R_j$ and we have $R\cap(R_i\cup A\cup R_j)\neq\emptyset$, as $R\cap S\neq\emptyset$.
Since $R$ is disjoint from $R_i$ and $R_j$, we have $R\cap A\neq\emptyset$.
\end{proof}

A point $x\in J$ is a \emph{neighbor} of $R_i$ if there is an arc in $J$ between $x$ and $R_i$ disjoint from all $R_1,\ldots,R_m$ except $R_i$.
It follows from Claim \ref{cla:inters} that each point in $J$ is a neighbor of at most two consecutive sets of $R_1,\ldots,R_m$.
For $1\leq i<m$, let $I_i$ denote the set of points in $J$ that are neighbors of $R_i$ and $R_{i+1}$.

\begin{claim}\label{cla:interval}
Any arc-connected subset of\/ $J$ intersects an interval of sets in the sequence\/ $R_1,I_1,R_2,\ldots,I_{m-1},R_m$.
\end{claim}

\begin{proof}
Let $X$ be an arc-connected subset of $J$.
First, we show that if $X$ intersects $R_i$ and $R_{i+1}$, then it also intersects $I_i$.
This is guaranteed by the compactness of $R_i$ and $R_{i+1}$.
Indeed, take a $\subset$-minimal arc in $J$ between $R_i$ and $R_{i+1}$.
By Claim \ref{cla:inters}, the interior of this arc is disjoint from all $R_1,\ldots,R_m$ and therefore must lie in $I_i$.

Now, let $i$ be the least index such that $X\cap(R_i\cup I_i)\neq\emptyset$ and $j$ be the greatest index such that $X\cap(I_{j-1}\cup R_j)\neq\emptyset$.
Let $A$ be an arc in $X$ between points $x_i\in X\cap(R_i\cup I_i)$ and $x_j\in X\cap(I_{j-1}\cup R_j)$.
Since $x_i$ is a neighbor of $R_i$, there is an arc $A_i\subset J$ between $x_i$ and $R_i$ disjoint from all $R_1,\ldots,R_m$ except $R_i$.
Similarly, since $x_j$ is a neighbor of $R_j$, there is an arc $A_j\subset J$ between $x_j$ and $R_j$ disjoint from all $R_1,\ldots,R_m$ except $R_j$.
There is an arc $\bar A\subset A_i\cup A\cup A_j$ between $R_i$ and $R_j$.
By Claim \ref{cla:inters}, $\bar A$ intersects all $R_{i+1},\ldots,R_{j-1}$.
But $R_{i+1},\ldots,R_{j-1}$ are disjoint from $A_i$ and $A_j$, hence they intersect~$A$.
\end{proof}

For convenience, define $I_0=I_m=\emptyset$.
For $D\in\calD$, define $\leftclip(D)=D\setminus I_i$ and $\rightclip(D)=D\setminus I_j$, where $i$ and $j$ are chosen so that $R_{i+1}$ is the first and $R_j$ is the last member of $\calR$ intersecting $D$ (see Figure \ref{fig:R-and-I-regions}).
This definition extends to families $\calM\subset\calD$: $\leftclip(\calM)=\{\leftclip(M)\colon M\in\calM\}$ and $\rightclip(\calM)=\{\rightclip(M)\colon M\in\calM\}$.

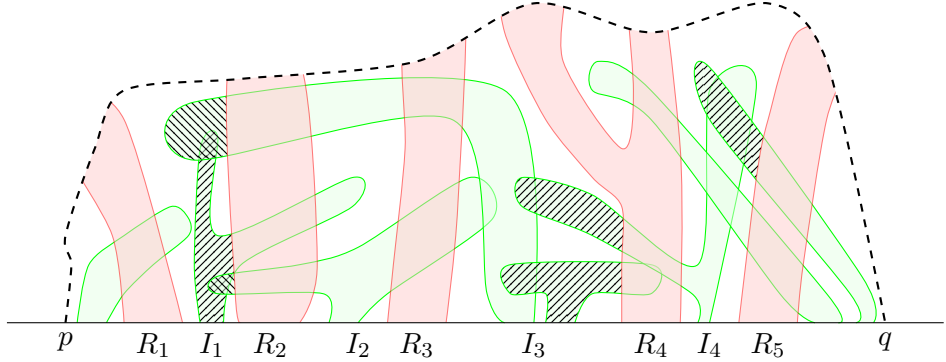
\begin{figure}[t]
\centering
\begin{tikzpicture}[scale=.77]
\coordinate (S0) at (1.0, 0.0);
\coordinate (S1) at (1.1, 1.0);
\coordinate (S2) at (1.0, 1.5);
\coordinate (S3) at (1.5, 3.0);
\coordinate (S4) at (2.0, 4.0);
\coordinate (S5) at (4.0, 4.2);
\coordinate (S6) at (7.0, 4.5);
\coordinate (S7) at (9.0, 5.5);
\coordinate (S8) at (11.0, 5.0);
\coordinate (S9) at (13.0, 5.5);
\coordinate (S10) at (14.0, 4.5);
\coordinate (S11) at (15.0, 0.0);

\coordinate (RA0) at (2.0, 0.0);
\coordinate (RA1) at (1.8, 1.5);
\coordinate (RA2) at (1.0, 3.0);
\coordinate (RA3) at (1.0, 4.0);
\coordinate (RA4) at (2.0, 3.5);
\coordinate (RA5) at (3.0, 0.0);

\coordinate (RB0) at (4.0, 0.0);
\coordinate (RB1) at (3.9, 0.5);
\coordinate (RB2) at (3.8, 4.5);
\coordinate (RB3) at (5.0, 4.5);
\coordinate (RB4) at (5.3, 1.0);
\coordinate (RB5) at (5.0, 0.0);

\coordinate (RC0) at (6.5, 0.0);
\coordinate (RC1) at (6.6, 1.0);
\coordinate (RC2) at (6.8, 3.0);
\coordinate (RC3) at (6.8, 5.0);
\coordinate (RC4) at (7.8, 6.0);
\coordinate (RC5) at (7.7, 2.0);
\coordinate (RC6) at (7.5, 0.0);

\coordinate (RD0) at (10.5, 0.0);
\coordinate (RD1) at (10.5, 1.0);
\coordinate (RD2) at (10.5, 2.0);
\coordinate (RD3) at (10.0, 2.5);
\coordinate (RD4) at (9.2, 3.5);
\coordinate (RD5) at (8.5, 5.0);
\coordinate (RD6) at (8.5, 6.0);
\coordinate (RD7) at (9.5, 6.0);
\coordinate (RD8) at (9.5, 4.5);
\coordinate (RD9) at (10.5, 3.0);
\coordinate (RD10) at (10.7, 4.0);
\coordinate (RD11) at (10.5, 7.0);
\coordinate (RD12) at (11.0, 7.0);
\coordinate (RD13) at (11.5, 3.0);
\coordinate (RD14) at (11.5, 0.0);

\coordinate (RE0) at (12.5, 0.0);
\coordinate (RE1) at (12.8, 2.5);
\coordinate (RE2) at (13.0, 3.5);
\coordinate (RE3) at (13.5, 5.0);
\coordinate (RE4) at (14.5, 5.0);
\coordinate (RE5) at (14.0, 3.5);
\coordinate (RE6) at (13.6, 1.0);
\coordinate (RE7) at (13.5, 0.0);

\coordinate (DA0) at (1.2, 0.0);
\coordinate (DA1) at (1.4, 1.0);
\coordinate (DA2) at (2.0, 1.5);
\coordinate (DA3) at (2.8, 2.0);
\coordinate (DA4) at (3.1, 1.8);
\coordinate (DA5) at (3.1, 1.5);
\coordinate (DA6) at (2.5, 1.0);
\coordinate (DA7) at (2.0, 0.5);
\coordinate (DA8) at (1.7, 0.0);

\coordinate (DB0) at (3.3, 0.0);
\coordinate (DB1) at (3.2, 1.0);
\coordinate (DB2) at (3.3, 3.0);
\coordinate (DB3) at (3.6, 3.2);
\coordinate (DB4) at (3.6, 1.5);
\coordinate (DB5) at (5.9, 2.5);
\coordinate (DB6) at (5.9, 2.0);
\coordinate (DB7) at (4.0, 1.0);
\coordinate (DB8) at (3.7, 0.0);

\coordinate (DC0) at (5.5, 0.0);
\coordinate (DC1) at (5.4, 0.5);
\coordinate (DC2) at (3.6, 0.5);
\coordinate (DC3) at (3.6, 0.8);
\coordinate (DC4) at (5.0, 1.0);
\coordinate (DC5) at (6.0, 1.3);
\coordinate (DC6) at (7.9, 2.5);
\coordinate (DC7) at (8.3, 1.8);
\coordinate (DC8) at (6.5, 0.5);
\coordinate (DC9) at (6.0, 0.0);

\coordinate (DD0) at (8.5, 0.0);
\coordinate (DD1) at (8.2, 0.5);
\coordinate (DD2) at (7.6, 3.5);
\coordinate (DD3) at (4.2, 3.4);
\coordinate (DD3) at (3.2, 2.8);
\coordinate (DD4) at (3.2, 3.8);
\coordinate (DD5) at (8.5, 4.0);
\coordinate (DD6) at (9.0, 0.0);

\coordinate (DG0) at (13.8, 0.0);
\coordinate (DG1) at (13.2, 0.5);
\coordinate (DG2) at (11.0, 3.5);
\coordinate (DG3) at (10.0, 4.0);
\coordinate (DG4) at (10.1, 4.5);
\coordinate (DG5) at (11.0, 4.0);
\coordinate (DG6) at (14.0, 0.5);
\coordinate (DG7) at (14.2, 0.0);

\coordinate (DH0) at (14.5, 0.0);
\coordinate (DH1) at (14.4, 0.5);
\coordinate (DH2) at (12.0, 3.5);
\coordinate (DH3) at (12.0, 4.4);
\coordinate (DH4) at (14.6, 0.8);
\coordinate (DH5) at (14.8, 0.0);

\coordinate (DE0) at (9.2, 0.0);
\coordinate (DE1) at (9.2, 0.5);
\coordinate (DE2) at (8.6, 0.5);
\coordinate (DE3) at (8.6, 1.0);
\coordinate (DE4) at (11.0, 1.0);
\coordinate (DE5) at (11.0, 0.5);
\coordinate (DE6) at (10.0, 0.5);
\coordinate (DE7) at (9.7, 0.0);

\coordinate (DF0) at (11.8, 0.0);
\coordinate (DF1) at (11.7, 0.5);
\coordinate (DF2) at (10.0, 1.5);
\coordinate (DF3) at (8.8, 2.0);
\coordinate (DF4) at (8.8, 2.5);
\coordinate (DF5) at (10.0, 2.2);
\coordinate (DF6) at (11.8, 1.0);
\coordinate (DF7) at (12.0, 4.0);
\coordinate (DF8) at (12.8, 4.1);
\coordinate (DF9) at (12.0, 0.0);

\fill[opacity=0.5, fill=white!90!green] plot[smooth, tension=0.5] coordinates
  {(DA0) (DA1) (DA2) (DA3) (DA4) (DA5) (DA6) (DA7) (DA8)};
\draw[color=green] plot[smooth, tension=0.5] coordinates
  {(DA0) (DA1) (DA2) (DA3) (DA4) (DA5) (DA6) (DA7) (DA8)};

\fill[opacity=0.5, fill=white!90!green] plot[smooth, tension=0.5] coordinates
  {(DB0) (DB1) (DB2) (DB3) (DB4) (DB5) (DB6) (DB7) (DB8)};
\begin{scope}
\begin{pgfinterruptboundingbox}
\path[clip] plot[smooth, tension=0.5] coordinates
  {(RB0) (RB1) (RB2) (RB3)} -- (100,100) -- (-100,100) -- (-100,0) -- cycle;
\end{pgfinterruptboundingbox}
\fill[pattern=north east lines] plot[smooth, tension=0.5] coordinates 
  {(DB0) (DB1) (DB2) (DB3) (DB4) (DB5) (DB6) (DB7) (DB8)};
\end{scope}
\draw[color=green] plot[smooth, tension=0.5] coordinates
  {(DB0) (DB1) (DB2) (DB3) (DB4) (DB5) (DB6) (DB7) (DB8)};

\fill[opacity=0.5, fill=white!90!green] plot[smooth, tension=0.5] coordinates
  {(DC0) (DC1) (DC2) (DC3) (DC4) (DC5) (DC6) (DC7) (DC8) (DC9)};
\begin{scope}
\begin{pgfinterruptboundingbox}
\path[clip] plot[smooth, tension=0.5] coordinates
  {(RB0) (RB1) (RB2) (RB3)} -- (100,100) -- (-100,100) -- (-100,0) -- cycle;
\end{pgfinterruptboundingbox}
\fill[pattern=north west lines] plot[smooth, tension=0.5] coordinates
  {(DC0) (DC1) (DC2) (DC3) (DC4) (DC5) (DC6) (DC7) (DC8) (DC9)};
\end{scope}
\draw[color=green] plot[smooth, tension=0.5] coordinates
  {(DC0) (DC1) (DC2) (DC3) (DC4) (DC5) (DC6) (DC7) (DC8) (DC9)};

\fill[opacity=0.5, fill=white!90!green] plot[smooth, tension=0.5] coordinates
  {(DD0) (DD1) (DD2) (DD3) (DD4) (DD5) (DD6)};
\begin{scope}
\begin{pgfinterruptboundingbox}
\path[clip] plot[smooth, tension=0.5] coordinates
  {(RB0) (RB1) (RB2) (RB3)} -- (100,100) -- (-100,100) -- (-100,0) -- cycle;
\end{pgfinterruptboundingbox}
\fill[pattern=north west lines] plot[smooth, tension=0.5] coordinates
  {(DD0) (DD1) (DD2) (DD3) (DD4) (DD5) (DD6)};
\end{scope}
\draw[color=green] plot[smooth, tension=0.5] coordinates
  {(DD0) (DD1) (DD2) (DD3) (DD4) (DD5) (DD6)};

\fill[opacity=0.5, fill=white!90!green] plot[smooth, tension=0.5] coordinates
  {(DE0) (DE1) (DE2) (DE3) (DE4) (DE5) (DE6) (DE7)};
\begin{scope}
\begin{pgfinterruptboundingbox}
\path[clip] plot[smooth, tension=0.5] coordinates
  {(RD0) (RD1) (RD2) (RD3) (RD4) (RD5) (RD6) (RD7)} -- (100,100) -- (-100,100) -- (-100,0) -- cycle;
\end{pgfinterruptboundingbox}
\fill[pattern=north east lines] plot[smooth, tension=0.5] coordinates
  {(DE0) (DE1) (DE2) (DE3) (DE4) (DE5) (DE6) (DE7)};
\end{scope}
\draw[color=green] plot[smooth, tension=0.5] coordinates
  {(DE0) (DE1) (DE2) (DE3) (DE4) (DE5) (DE6) (DE7)};

\fill[opacity=0.5, fill=white!90!green] plot[smooth, tension=0.5] coordinates
  {(DF0) (DF1) (DF2) (DF3) (DF4) (DF5) (DF6) (DF7) (DF8) (DF9)};
\begin{scope}
\begin{pgfinterruptboundingbox}
\path[clip] plot[smooth, tension=0.5] coordinates
  {(RD0) (RD1) (RD2) (RD3) (RD4) (RD5) (RD6) (RD7)} -- (100,100) -- (-100,100) -- (-100,0) -- cycle;
\end{pgfinterruptboundingbox}
\fill[pattern=north east lines] plot[smooth, tension=0.5] coordinates
  {(DF0) (DF1) (DF2) (DF3) (DF4) (DF5) (DF6) (DF7) (DF8) (DF9)};
\end{scope}
\draw[color=green] plot[smooth, tension=0.5] coordinates
  {(DF0) (DF1) (DF2) (DF3) (DF4) (DF5) (DF6) (DF7) (DF8) (DF9)};

\fill[opacity=0.5, fill=white!90!green] plot[smooth, tension=0.5] coordinates
  {(DG0) (DG1) (DG2) (DG3) (DG4) (DG5) (DG6) (DG7)};
\draw[color=green] plot[smooth, tension=0.5] coordinates
  {(DG0) (DG1) (DG2) (DG3) (DG4) (DG5) (DG6) (DG7)};

\fill[opacity=0.5, fill=white!90!green] plot[smooth, tension=0.5] coordinates
  {(DH0) (DH1) (DH2) (DH3) (DH4) (DH5)};
\begin{scope}
\begin{pgfinterruptboundingbox}
\path[clip] plot[smooth, tension=0.5] coordinates
  {(RE0) (RE1) (RE2) (RE3) (RE4)} -- (100,100) -- (-100,100) -- (-100,0) -- cycle;
\end{pgfinterruptboundingbox}
\fill[pattern=north east lines] plot[smooth, tension=0.5] coordinates
  {(DH0) (DH1) (DH2) (DH3) (DH4) (DH5)};
\end{scope}
\draw[color=green] plot[smooth, tension=0.5] coordinates
  {(DH0) (DH1) (DH2) (DH3) (DH4) (DH5)};

\begin{scope}
\path[clip] plot[smooth, tension=0.5] coordinates
  {(S0) (S1) (S2) (S3) (S4) (S5) (S6) (S7) (S8) (S9) (S10) (S11)};
\draw[opacity=0.5, color=red, fill=white!80!red] plot[smooth, tension=0.5] coordinates
  {(RA0) (RA1) (RA2) (RA3) (RA4) (RA5)};
\draw[opacity=0.5, color=red, fill=white!80!red] plot[smooth, tension=0.5] coordinates
  {(RB0) (RB1) (RB2) (RB3) (RB4) (RB5)};
\draw[opacity=0.5, color=red, fill=white!80!red] plot[smooth, tension=0.5] coordinates
  {(RC0) (RC1) (RC2) (RC3) (RC4) (RC5) (RC6)};
\draw[opacity=0.5, color=red, fill=white!80!red] plot[smooth, tension=0.5] coordinates
  {(RD0) (RD1) (RD2) (RD3) (RD4) (RD5) (RD6) (RD7) (RD8) (RD9) (RD10) (RD11) (RD12) (RD13) (RD14)};
\draw[opacity=0.5, color=red, fill=white!80!red] plot[smooth, tension=0.5] coordinates
  {(RE0) (RE1) (RE2) (RE3) (RE4) (RE5) (RE6) (RE7)};
\end{scope}

\path[draw, thick, dashed] plot[smooth, tension=0.5] coordinates
  {(S0) (S1) (S2) (S3) (S4) (S5) (S6) (S7) (S8) (S9) (S10) (S11)};

\draw (S0) node[below] {$p$};
\draw (S11) node[below] {$q$};

\draw (2.5, 0.0) node[below] {$R_1$};
\draw (3.5, 0.0) node[below] {$I_1$};
\draw (4.5, 0.0) node[below] {$R_2$};
\draw (6.0, 0.0) node[below] {$I_2$};
\draw (7.0, 0.0) node[below] {$R_3$};
\draw (9.0, 0.0) node[below] {$I_3$};
\draw (11.0, 0.0) node[below] {$R_4$};
\draw (12.0, 0.0) node[below] {$I_4$};
\draw (13.0, 0.0) node[below] {$R_5$};

\draw (0,0) -- (16,0) {};
\end{tikzpicture}
\caption{A dashed arc $S$, sets from $\calR$ spanned between the baseline and $S$, and sets from $\calD$ with $D\setminus\leftclip(D)$ marked for each $D\in\calD$.}
\label{fig:R-and-I-regions}
\end{figure}

\begin{claim}\label{cla:simple}
$\leftclip(\calD)$ and\/ $\rightclip(\calD)$ are simple families of compact arc-connected sets.
\end{claim}

\begin{proof}
We present the proof only for $\leftclip(\calD)$, as for $\rightclip(\calD)$ it is analogous.
The sets $I_1,\ldots,I_{m-1}$ are open in $J$, as they are arc-connected components of the set $J\setminus\bigcup\calR$, which is open in $J$.
Each member of $\leftclip(\calD)$ is a difference of a compact set in $\calD$ and one of $I_1,\ldots,I_{m-1}$ and thus is compact as well.

To prove that $\leftclip(\calD)$ is simple and consists of arc-connected sets, we need to show that $\bigcap\leftclip(\calM)$ is arc-connected for any $\calM\subset\calD$.
Let $x,y\in\bigcap\leftclip(\calM)\subset\bigcap\calM$.
Since $\bigcap\calM$ is arc-connected, Lemma \ref{lem:curve} provides us with an arc $A\subset\bigcap\calM$ between $x$ and $y$ that is simple with respect to $\calR$.
It suffices to show $A\subset\leftclip(M)$ for each $M\in\calM$.
To this end, fix $M\in\calM$ and let $R_i$ be the $\prec$-least member of $\calR$ intersecting $M$. 
Suppose there is a point $z\in A\cap(M\setminus\leftclip(M))\subset I_{i-1}$.
Since $x,y\in\leftclip(M)\subset\bigcup_{j=i}^m(R_j\cup I_j)$, it follows from Claim \ref{cla:interval} that the parts of $A$ from $x$ to $z$ and from $z$ to $y$ intersect $R_i$.
This and $z\notin R_i$ contradict the simplicity of $A$ with respect to $R_i$.
\end{proof}

\begin{claim}\label{cla:clique}
$\leftclip(\calD)$ and\/ $\rightclip(\calD)$ have clique number at most\/ $k-1$.
\end{claim}

\begin{proof}
Again, we present the proof only for $\leftclip(\calD)$.
Let $\calK$ be a clique in $\leftclip(\calD)$.
By Claim \ref{cla:simple}, each member of $\calK$ is arc-connected.
Therefore, by Claim \ref{cla:interval}, each member of $\calK$ intersects an interval of sets in the sequence $R_1,I_1,\ldots,R_{m-1},I_{m-1},R_m$, the first set in this interval being of the form $R_j$.
Since the members of $\calK$ pairwise intersect, these intervals also pairwise intersect, which implies that they all contain a common $R_j$.
Thus $\calK\cup\{R_j\}$ is a clique, and $\omega(\calR\cup\calD)\leq k$ yields $|\calK|\leq k-1$.
\end{proof}

Recall that we need to show $\chi(\calD)\leq 8\xi_{k-1}^2$.
Define
\begin{itemize}
\item $\calD^L=\{X\in\calD\colon\base(\leftclip(X))\neq\emptyset\}$,
\item $\calD^R=\{X\in\calD\colon\base(\rightclip(X))\neq\emptyset\}$.
\end{itemize}
Since each member of $\calD$ intersects at least one of $R_1,\ldots,R_m$, we have $\calD^L\cup\calD^R=\calD$.
Therefore, it is enough to show that $\chi(\calD^L)\leq 4\xi_{k-1}^2$ and $\chi(\calD^R)\leq 4\xi_{k-1}^2$.
We only present the proof of $\chi(\calD^R)\leq 4\xi_{k-1}^2$, as the proof of the other inequality is analogous.

By Claims \ref{cla:simple} and \ref{cla:clique}, $\rightclip(\calD^R)$ is a grounded family with clique number at most $k-1$.
This and the induction hypothesis yield $\chi(\rightclip(\calD^R))\leq\xi_{k-1}$.
We fix a coloring $\phi^R$ of $\calD^R$ with $\xi_{k-1}$ colors so that $\phi^R(X)\neq\phi^R(Y)$ for any $X,Y\in\calD^R$ with $\rightclip(X)\cap\rightclip(Y)\neq\emptyset$.
Let $\calM\subset\calD^R$ be a family of sets having the same color in $\phi^R$.
In particular, we have $\rightclip(X)\cap\rightclip(Y)=\emptyset$ for any $X,Y\in\calM$.
It remains to prove that $\chi(\calM)\leq 4\xi_{k-1}$.

We show how to construct a coloring $\phi^L$ of $\calM$ with $\xi_{k-1}$ colors such that $\leftclip(X)\cap\leftclip(Y)=\emptyset$ for any $X,Y\in\calM$ with $\phi^L(X)=\phi^L(Y)$.
We exploit the fact that members of $\calM$ have pairwise disjoint intersections with each $R_i$ to simplify the topology of $\calM$ and $R_1,\ldots,R_m$.
Recall that $S$ is an arc with $\base(S)=\{p,q\}$.
For $1\leq i\leq m$, by Lemma \ref{lem:curve}, there is an arc $Q_i\subset R_i$ between $\base(R_i)$ and $R_i\cap S$ that is simple with respect to $\calM$.
We assume without loss of generality that $\base(Q_i)=\{u_i\}$ and $Q_i\cap S=\{v_i\}$.
The points $v_1,\ldots,v_m$ occur in this order on $S$ as moving from $p$ to $q$.
Moreover, the arcs $Q_1,\ldots,Q_m$ partition $J$ into $m+1$ sets $J_0,\ldots,J_m$, each homeomorphic to a closed disc, so that $J_{i-1}\cap J_i=Q_i$ for $1\leq i\leq m$.
It is clear that each arc-connected subset of $J$ intersects an interval of sets in the sequence $J_0,Q_1,J_1,\ldots,Q_m,J_m$.
Since $Q_1,\ldots,Q_m$ are simple with respect to $\calM$, so are $J_0,\ldots,J_m$.

Since the sets $J_0,\ldots,J_m$ are homeomorphic to a closed disc and so are rectangles with bottom sides $\base(J_0),\ldots,\base(J_m)$, there are homeomorphisms $\mu_0,\ldots,\mu_m$ such that
\begin{itemize}
\item $\mu_i$ is constant on $\base(J_i)$ and maps $J_i$ onto a rectangle with bottom side $\base(J_i)$ for $0\leq i\leq m$,
\item $\mu_{i-1}$ and $\mu_i$ agree on $Q_i$ for $1\leq i\leq m$.
\end{itemize}
Thus $\mu_0\cup\cdots\cup\mu_m$ is a homeomorphism between $J$ and a rectangle with bottom side $\base(J)$, and it extends to a homeomorphism $\mu$ of $\setR\times[0,+\infty)$ whose restriction to each $J_i$ is $\mu_i$.
Let $\tau_0,\ldots,\tau_m$ be horizontal translations such that $\tau_0(u)\prec\cdots\prec\tau_m(u)$ for a point $u$ on the baseline.
Let $x_1,\ldots,x_m$ be the $x$-coordinates of the points $u_1,\ldots,u_m$, respectively, so that $u_i=(x_i,0)$ for $1\leq i\leq m$.
Define
\begin{itemize}
\item $\hat J_0=\mu^{-1}((-\infty,x_1]\times[0,+\infty))$,
\item $\hat J_i=\mu^{-1}([x_i,x_{i+1}]\times[0,+\infty))$ for $1\leq i<m$,
\item $\hat J_m=\mu^{-1}([x_m,+\infty)\times[0,+\infty))$,
\item $\hat Q_i=\mu^{-1}(\{x_i\}\times[0,+\infty))$ for $1\leq i\leq m$,
\item $\sigma_i=\tau_i\circ\mu$ for $0\leq i\leq m$.
\end{itemize}
Note that $J_i\subset\hat J_i$ and $Q_i\subset\hat Q_i$.
For a set $X$, define
\begin{alignat*}{10}
X^\star=\sigma_0(X\cap\hat J_0)&\cup[\sigma_0&&(X\cap\hat Q_1&&),\sigma_1&&(X\cap\hat Q_1&&)]\cup\sigma_1&&(X\cap\hat J_1&&)\cup\cdots\\
&\cup[\sigma_{m-1}&&(X\cap\hat Q_m&&),\sigma_m&&(X\cap\hat Q_m&&)]\cup\sigma_m&&(X\cap\hat J_m&&)
\end{alignat*}
where $[Y,Z]$ denotes the rectangle with left side $Y$ and right side $Z$.
It is clear that the map $X\mapsto X^\star$ preserves compactness and arc-connectedness and is compatible with unions and intersections.
In particular, $\calM^\star=\{X^\star\colon X\in\calM\}$ is a grounded family with intersection graph isomorphic to that of $\calM$ (see Figure \ref{fig:transformation}).

\begin{figure}[t]
\centering
\begin{tikzpicture}[scale=.77]
\begin{scope}
\clip plot[smooth, tension=.7] coordinates {(-5.5,-1.5) (-4.5,3.5) (-1.5,3.5) (4.5,3) (5.5,-1.5)} -- cycle;
\fill[white!80!red] plot[smooth, tension=.7] coordinates {(-4.5,-1.5) (-5,1.5) (-6,3.5) (-4,4) (-3.5,-1.5)} -- cycle;
\draw[red] plot[smooth, tension=.7] coordinates {(-4.5,-1.5) (-5,1.5) (-6,3.5) (-4,4) (-3.5,-1.5)};
\fill[white!80!red] plot[smooth, tension=.7] coordinates {(-1,-1.5) (-1.5,4) (0,4) (0,-1.5)} -- cycle;
\draw[red] plot[smooth, tension=.7] coordinates {(-1,-1.5) (-1.5,4) (0,4) (0,-1.5)};
\fill[white!80!red] plot[smooth, tension=.7] coordinates {(2.5,-1.5) (2.5,4) (4,4) (3.5,1.5) (3.5,-1.5)} -- cycle;
\draw[red] plot[smooth, tension=.7] coordinates {(2.5,-1.5) (2.5,4) (4,4) (3.5,1.5) (3.5,-1.5)};
\draw[red,dashed] plot[smooth, tension=.7] coordinates {(-4,-1.5) (-4.5,2) (-5.5,3.5)};
\draw[red,dashed] plot[smooth, tension=.7] coordinates {(-0.5,-1.5) (-0.6,1.8) (-0.8,4)};
\draw[red,dashed] plot[smooth, tension=.7] coordinates {(3,-1.5) (3,2.5) (3.4,4)};
\fill[opacity=0.5,color=white!80!green] plot[smooth, tension=.7] coordinates {(-3,-1.5) (-3.2,1.2) (-4,3) (1,2.6) (1.4,1.8) (-2,2) (-2.5,-1.5)} -- cycle;
\draw[green] plot[smooth, tension=.7] coordinates {(-3,-1.5) (-3.2,1.2) (-4,3) (1,2.6) (1.4,1.8) (-2,2) (-2.5,-1.5)};
\fill[opacity=0.5,color=white!80!green] plot[smooth, tension=.7] coordinates {(5,-1.5) (4.5,2) (1.8,2.6) (1.6,1.8) (3.5,1.5) (4.5,-1.5)} -- cycle;
\draw[green] plot[smooth, tension=.7] coordinates {(5,-1.5) (4.5,2) (1.8,2.6) (1.6,1.8) (3.5,1.5) (4.5,-1.5)};
\fill[opacity=0.5,color=white!80!green] plot[smooth, tension=.7] coordinates {(0.5,-1.5) (1,1) (4,1) (4,0.5) (1.5,0.5) (1,-1.5)} -- cycle;
\draw[green] plot[smooth, tension=.7] coordinates {(0.5,-1.5) (1,1) (4,1) (4,0.5) (1.5,0.5) (1,-1.5)};
\fill[opacity=0.5,color=white!80!green] plot[smooth, tension=.7] coordinates {(1.5,-1.5) (2,0) (3.6,0) (3.6,-0.6) (2.5,-0.5) (2,-1.5)} -- cycle;
\draw[green] plot[smooth, tension=.7] coordinates {(1.5,-1.5) (2,0) (3.6,0) (3.6,-0.6) (2.5,-0.5) (2,-1.5)};
\fill[opacity=0.5,color=white!80!green] plot[smooth, tension=.7] coordinates {(-2,-1.5) (-1.5,1) (1.5,0.5) (2,-0.5) (-0.5,0) (-1.5,-1.5)} -- cycle;
\draw[green] plot[smooth, tension=.7] coordinates {(-2,-1.5) (-1.5,1) (1.5,0.5) (2,-0.5) (-0.5,0) (-1.5,-1.5)};
\end{scope}
\draw[dashed] plot[smooth, tension=.7] coordinates {(-5.5,-1.5) (-4.5,3.5) (-1.5,3.5) (4.5,3) (5.5,-1.5)};
\draw (-6,-1.5) -- (6,-1.5);
\node[below] at (-5.5,-1.5) {$p$};
\node[below] at (-4,-1.5) {$u_1$};
\node[below] at (-0.5,-1.5) {$u_2$};
\node[below] at (3,-1.5) {$u_3$};
\node[below] at (5.5,-1.5) {$q$};
\node[above left] at (-4.8,2.8) {$v_1$};
\node[above] at (-0.7,3.5) {$v_2$};
\node[above right] at (3.2,3.5) {$v_3$};
\end{tikzpicture}\\[2ex]
\begin{tikzpicture}[scale=0.7]
\fill[white!80!red] (-4,-1.5) rectangle (-3,3.5);
\draw[red,dashed] (-4,-1.5) -- (-4,3.5);
\begin{scope}[xshift=1cm]
\begin{scope}
\clip (-4,-1.5) rectangle (-0.5,3.5);
\fill[opacity=0.5,color=white!80!green] plot[smooth, tension=.7] coordinates {(-3,-1.5) (-3.1,1.2) (-3.4,3) (-0.5,2.9) (1.4,2.5) (1.3,1.9) (-0.5,2.3) (-2.2,1.7) (-2.5,-1.5)} -- cycle;
\draw[green] plot[smooth, tension=.7] coordinates {(-3,-1.5) (-3.1,1.2) (-3.4,3) (-0.5,2.9) (1.4,2.5) (1.3,1.9) (-0.5,2.3) (-2.2,1.7) (-2.5,-1.5)};
\fill[opacity=0.5,color=white!80!green] plot[smooth, tension=.7] coordinates {(-2,-1.5) (-1.6,0.5) (-0.5,1.1) (1.5,0.5) (2,-0.5) (-0.5,0) (-1.5,-1.5)} -- cycle;
\fill[pattern=north west lines] plot[smooth, tension=.7] coordinates {(-2,-1.5) (-1.6,0.5) (-0.5,1.1) (1.5,0.5) (2,-0.5) (-0.5,0) (-1.5,-1.5)} -- cycle;
\draw[green] plot[smooth, tension=.7] coordinates {(-2,-1.5) (-1.6,0.5) (-0.5,1.1) (1.5,0.5) (2,-0.5) (-0.5,0) (-1.5,-1.5)};
\end{scope}
\fill[white!80!red] (-0.5,-1.5) rectangle (0.5,3.5);
\begin{scope}
\clip (-0.5,-1.5) rectangle (0.5,3.5);
\fill[opacity=0.5,color=white!80!green] (-0.5,2.3) rectangle (0.5,2.9);
\fill[opacity=0.5,color=white!80!green] (-0.5,0) rectangle (0.5,1.1);
\fill[pattern=north west lines] (-0.5,0) rectangle (0.5,1.1);
\end{scope}
\draw[green,line cap=round] (-0.5,2.9) -- (0.5,2.9);
\draw[green,line cap=round] (-0.5,2.3) -- (0.5,2.3);
\draw[green,line cap=round] (-0.5,1.1) -- (0.5,1.1);
\draw[green,line cap=round] (-0.5,0) -- (0.5,0);
\draw[red,dashed] (-4,-1.5) -- (-4,3.5);
\draw[red,dashed] (-0.5,-1.5) -- (-0.5,3.5);
\end{scope}
\begin{scope}[xshift=2cm]
\begin{scope}
\clip (-0.5,-1.5) rectangle (3,3.5);
\fill[opacity=0.5,color=white!80!green] plot[smooth, tension=.7] coordinates {(0.5,-1.5) (0.9,0.9) (3,1.2) (4.2,1) (4.1,0.5) (3,0.7) (1.4,0.4) (1,-1.5)} -- cycle;
\fill[pattern=north west lines] plot[smooth, tension=.7] coordinates {(0.5,-1.5) (0.9,0.9) (3,1.2) (4.2,1) (4.1,0.5) (3,0.7) (1.4,0.4) (1,-1.5)} -- cycle;
\draw[green] plot[smooth, tension=.7] coordinates {(0.5,-1.5) (0.9,0.9) (3,1.2) (4.2,1) (4.1,0.5) (3,0.7) (1.4,0.4) (1,-1.5)};
\fill[opacity=0.5,color=white!80!green] plot[smooth, tension=.7] coordinates {(1.5,-1.5) (2,0) (3,0.2) (3.6,0) (3.6,-0.6) (3,-0.3) (2.5,-0.5) (2,-1.5)} -- cycle;
\fill[pattern=north west lines] plot[smooth, tension=.7] coordinates {(1.5,-1.5) (2,0) (3,0.2) (3.6,0) (3.6,-0.6) (3,-0.3) (2.5,-0.5) (2,-1.5)} -- cycle;
\draw[green] plot[smooth, tension=.7] coordinates {(1.5,-1.5) (2,0) (3,0.2) (3.6,0) (3.6,-0.6) (3,-0.3) (2.5,-0.5) (2,-1.5)};
\fill[opacity=0.5,color=white!80!green] plot[smooth, tension=.7] coordinates {(5,-1.5) (4.6,1.6) (3,2.4) (1.5,2.4) (1.6,1.8) (3,1.7) (3.9,0.9) (4.5,-1.5)} -- cycle;
\fill[pattern=north west lines] plot[smooth, tension=.7] coordinates {(5,-1.5) (4.6,1.6) (3,2.4) (1.5,2.4) (1.6,1.8) (3,1.7) (3.9,0.9) (4.5,-1.5)} -- cycle;
\draw[green] plot[smooth, tension=.7] coordinates {(5,-1.5) (4.6,1.6) (3,2.4) (1.5,2.4) (1.6,1.8) (3,1.7) (3.9,0.9) (4.5,-1.5)};
\fill[opacity=0.5,color=white!80!green] plot[smooth, tension=.7] coordinates {(-3,-1.5) (-3.1,1.2) (-3.4,3) (-0.5,2.9) (1.4,2.5) (1.3,1.9) (-0.5,2.3) (-2.2,1.7) (-2.5,-1.5)} -- cycle;
\draw[green] plot[smooth, tension=.7] coordinates {(-3,-1.5) (-3.1,1.2) (-3.4,3) (-0.5,2.9) (1.4,2.5) (1.3,1.9) (-0.5,2.3) (-2.2,1.7) (-2.5,-1.5)};
\fill[opacity=0.5,color=white!80!green] plot[smooth, tension=.7] coordinates {(-2,-1.5) (-1.6,0.5) (-0.5,1.1) (1.5,0.5) (2,-0.5) (-0.5,0) (-1.5,-1.5)} -- cycle;
\draw[green] plot[smooth, tension=.7] coordinates {(-2,-1.5) (-1.6,0.5) (-0.5,1.1) (1.5,0.5) (2,-0.5) (-0.5,0) (-1.5,-1.5)};
\end{scope}
\fill[white!80!red] (3,-1.5) rectangle (4,3.5);
\begin{scope}
\clip (3,-1.5) rectangle (4,3.5);
\fill[opacity=0.5,color=white!80!green] (3,0.7) rectangle (4,1.2);
\fill[pattern=north west lines] (3,0.7) rectangle (4,1.2);
\fill[opacity=0.5,color=white!80!green] (3,-0.3) rectangle (4,0.2);
\fill[pattern=north west lines] (3,-0.3) rectangle (4,0.2);
\fill[opacity=0.5,color=white!80!green] (3,1.7) rectangle (4,2.4);
\fill[pattern=north west lines] (3,1.7) rectangle (4,2.4);
\end{scope}
\draw[green,line cap=round] (3,1.2) -- (4,1.2);
\draw[green,line cap=round] (3,0.7) -- (4,0.7);
\draw[green,line cap=round] (3,0.2) -- (4,0.2);
\draw[green,line cap=round] (3,-0.3) -- (4,-0.3);
\draw[green,line cap=round] (3,2.4) -- (4,2.4);
\draw[green,line cap=round] (3,1.7) -- (4,1.7);
\draw[green!50!black,thick] (-0.7,-1.5) -- (-0.7,0.3) arc (180:90:0.2);
\draw[red,dashed] (-0.5,-1.5) -- (-0.5,3.5);
\draw[red,dashed] (3,-1.5) -- (3,3.5);
\end{scope}
\begin{scope}[xshift=3cm]
\begin{scope}
\clip (3,-1.5) rectangle (5.5,3.5);
\fill[opacity=0.5,color=white!80!green] plot[smooth, tension=.7] coordinates {(0.5,-1.5) (0.9,0.9) (3,1.2) (4.2,1) (4.1,0.5) (3,0.7) (1.4,0.4) (1,-1.5)} -- cycle;
\draw[green] plot[smooth, tension=.7] coordinates {(0.5,-1.5) (0.9,0.9) (3,1.2) (4.2,1) (4.1,0.5) (3,0.7) (1.4,0.4) (1,-1.5)};
\fill[opacity=0.5,color=white!80!green] plot[smooth, tension=.7] coordinates {(1.5,-1.5) (2,0) (3,0.2) (3.6,0) (3.6,-0.6) (3,-0.3) (2.5,-0.5) (2,-1.5)} -- cycle;
\draw[green] plot[smooth, tension=.7] coordinates {(1.5,-1.5) (2,0) (3,0.2) (3.6,0) (3.6,-0.6) (3,-0.3) (2.5,-0.5) (2,-1.5)};
\fill[opacity=0.5,color=white!80!green] plot[smooth, tension=.7] coordinates {(5,-1.5) (4.6,1.6) (3,2.4) (1.5,2.4) (1.6,1.8) (3,1.7) (3.9,0.9) (4.5,-1.5)} -- cycle;
\draw[green] plot[smooth, tension=.7] coordinates {(5,-1.5) (4.6,1.6) (3,2.4) (1.5,2.4) (1.6,1.8) (3,1.7) (3.9,0.9) (4.5,-1.5)};
\end{scope}
\draw[green,line cap=round] (3,0.7) -- (3,1.2);
\draw[green,line cap=round] (3,-0.3) -- (3,0.2);
\draw[green,line cap=round] (3,1.7) -- (3,2.4);
\draw[green!50!black,thick] (2.8,-1.5) -- (2.8,-0.3) arc (180:90:0.2);
\draw[green!50!black,thick] (2.6,-1.5) -- (2.6,0.55) arc (180:90:0.4);
\draw[red,dashed] (3,-1.5) -- (3,3.5);
\end{scope}
\draw[dashed] (-5.5,-1.5) -- (-5.5,3.5) -- (8.5,3.5) -- (8.5,-1.5);
\draw (-6,-1.5) -- (9,-1.5);
\node[below] (a) at (-6.5,-2.5) {$\sigma_0(p)$};
\node[below] (b) at (-4.5,-2.5) {$\sigma_0(u_1)$};
\node[below] (c) at (-2.5,-2.5) {$\sigma_1(u_1)$};
\node[below] (d) at (0,-2.5) {$\sigma_1(u_2)$};
\node[below] (e) at (2,-2.5) {$\sigma_2(u_2)$};
\node[below] (f) at (4.5,-2.5) {$\sigma_2(u_3)$};
\node[below] (g) at (6.5,-2.5) {$\sigma_3(u_3)$};
\node[below] (h) at (8.5,-2.5) {$\sigma_3(q)$};
\begin{scope}[>=latex,->,shorten >=1pt]
\path (a) edge (-5.5,-1.5);
\path (b) edge (-4,-1.5);
\path (c) edge (-3,-1.5);
\path (d) edge (0.5,-1.5);
\path (e) edge (1.5,-1.5);
\path (f) edge (5,-1.5);
\path (g) edge (6,-1.5);
\path (h) edge (8.5,-1.5);
\end{scope}
\end{tikzpicture}
\caption{The transformation $X\mapsto X^\star$;\enspace top: a family $\calM$ before transformation;\enspace bottom: the families $\calM^\star$ (including the marked regions) and $\calM^+$ (excluding the marked regions), and connections of the sets in $\calM^\star$ to the baseline.}
\label{fig:transformation}
\end{figure}
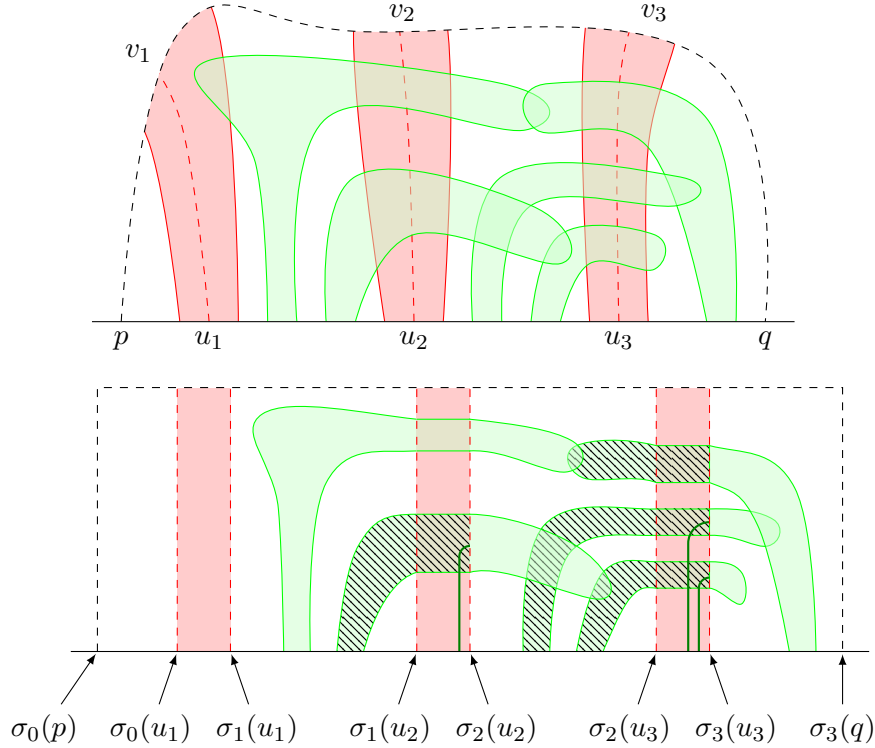

In the remainder of the proof, we will deal with $\calM^\star$ and $R_1^\star,\ldots,R_m^\star$, but for simplicity we relabel them to $\calM$ and $R_1,\ldots,R_m$, respectively.
We also relabel $I_0^\star,\ldots,I_m^\star$ to $I_0,\ldots,I_m$, $Q_1^\star,\ldots,Q_m^\star$ to $Q_1,\ldots,Q_m$, $S^\star$ to $S$, $J^\star$ to $J$, and $\sigma_0(J_0),\ldots,\sigma_m(J_m)$ to $J_0,\ldots,J_m$.
The following properties clearly follow:
\begin{itemize}
\item $J_0,\ldots,J_m$ are pairwise disjoint rectangles,
\item $Q_i$ is a rectangle whose left side is the right side of $J_{i-1}$ and whose right side is the left side of $J_i$, for $1\leq i\leq m$,
\item every arc-connected subset of $J$ intersects an interval of sets in the sequence $J_0,Q_1,J_1,\ldots,Q_m,J_m$,
\item $Q_i\subset R_i$ for $1\leq i\leq m$,
\item the intersection of any member of $\calM$ with any $J_i$ is arc-connected,
\item the intersection of any member of $\calM$ with any $Q_i$ is a rectangle or horizontal segment spanning the entire width of $Q_i$. 
\end{itemize}

\begin{claim}\label{cla:attach}
There is a grounded family\/ $\calM'$ with intersection graph isomorphic to that of\/ $\leftclip(\calM)$.
\end{claim}

\begin{proof}
Let $\calM_i=\{X\in\calM\colon X\cap R_i\neq\emptyset\}$ for $1\leq i\leq m$.
It follows that $\leftclip(X)=X\setminus I_i$ for every $X\in\calM_{i+1}\setminus\calM_i$ and $1\leq i<m$.
We can assume without loss of generality that each member of $\calM_{i+1}\setminus\calM_i$ intersects $Q_{i+1}$ for $1\leq i<m$, as those that do not are isolated vertices in the intersection graph of $\leftclip(\calM)$ and thus do not influence the existence of $\calM'$.

For $X\in\calM_{i+1}\setminus\calM_i$ and $1\leq i<m$, let $X^+$ denote the part of $X$ to the right of $Q_{i+1}$ including the right side of $Q_{i+1}$, that is, $X^+=X\cap(J_{i+1}\cup Q_{i+2}\cup J_{i+2}\cup\cdots\cup Q_m\cup J_m)$.
It follows that $X^+\subset\leftclip(X)\subset X^+\cup R_{i+1}$ and $X^+$ is compact and arc-connected.
For $X\in\calM_1$, let $X^+=X=\leftclip(X)$.
Let $\calM^+=\{X^+\colon X\in\calM\}$.
Since the intersections of the members of $\leftclip(\calM)$ with each $R_{i+1}$ are pairwise disjoint, the intersection graphs of $\calM^+$ and $\leftclip(\calM)$ are isomorphic.
We show how to extend the sets $X^+$ for which $\base(X^+)=\base(\leftclip(X))=\emptyset$ to connect them to the baseline without creating any new intersections (see Figure \ref{fig:transformation}).

Let $1\leq i<m$, $X\in\calM_{i+1}\setminus\calM_i$, and $\base(X^+)=\emptyset$.
Thus $\base(X)\subset\base(J_i)$.
For every $Y\in\calM_i$, it is an immediate consequnce of the Jordan curve theorem and arc-connectedness of $X\cap J_i$ and $(Y\cup R_i)\cap J_i$ that $Y\cap Q_{i+1}$ is empty or lies above $X\cap Q_{i+1}$.
Therefore, we can connect $X^+$ to $\base(Q_{i+1})$ by an arc inside $Q_{i+1}$ that is disjoint from any other $Y^+\in\calM^+$.
Moreover, all these arcs for $X\in\calM_{i+1}\setminus\calM_i$ with $\base(X^+)=\emptyset$ can be drawn so that they are pairwise disjoint.
Doing so for all $i$ with $1\leq i<m$, we transform $\calM^+$ into a grounded family $\calM'$ with intersection graph isomorphic to the intersection graph of $\calM^+$ and hence of $\leftclip(\calM)$.
\end{proof}

Claim \ref{cla:attach} allows us to use the induction hypothesis on $\leftclip(\calM)$ to obtain a coloring $\phi^L$ of $\calM$ with $\xi_{k-1}$ colors such that $\leftclip(X)\cap\leftclip(Y)=\emptyset$ for any $X,Y\in\calM$ with $\phi^L(X)=\phi^L(Y)$.
Let $\calN\subset\calM$ be a family of sets having the same color in $\phi^L$.
In particular, we have $\leftclip(X)\cap\leftclip(Y)=\emptyset$ and $\rightclip(X)\cap\rightclip(Y)=\emptyset$ for any $X,Y\in\calN$.
The following claim completes the proof of Lemma \ref{lem:dist2}.
The planarity argument used in the proof applies the ideas of McGuinness \cite{McG00}.

\begin{claim}\label{cla:final}
$\chi(\calN)\leq 4$.
\end{claim}

\begin{proof}
Define
\begin{itemize}
\item $\calN^L=\{X\in\calN\colon X\setminus\leftclip(X)\neq\emptyset\}$,
\item $\calN^R=\{X\in\calN\colon X\setminus\rightclip(X)\neq\emptyset\}$,
\item $\calN^L_i=\{X\in\calN^L\colon X\setminus\leftclip(X)\subset J_i\}$ for $1\leq i<m$,
\item $\calN^R_i=\{X\in\calN^R\colon X\setminus\rightclip(X)\subset J_i\}$ for $1\leq i<m$,
\item $\calN_i=\calN^L_i\cup\calN^R_i$,
\item $C^L_X$ to be the arc-connected component of $(\bigcup\calN_i)\cap J_i$ containing $X\setminus\leftclip(X)$, for $X\in\calN^L_i$ and $1\leq i<m$,
\item $C^R_X$ to be the arc-connected component of $(\bigcup\calN_i)\cap J_i$ containing $X\setminus\rightclip(X)$, for $X\in\calN^R_i$ and $1\leq i<m$,
\item $\calC$ to be the family of arc-connected components of $(\bigcup\calN_i)\cap J_i$ for $1\leq i<m$, that is, $\calC=\{C^L_X\colon X\in\calN^L\}\cup\{C^R_X\colon X\in\calN^R\}$.
\end{itemize}
For $X\in\calN_i$ and $1\leq i<m$, the set $X\cap J_i$ is arc-connected and thus entirely contained in one member of $\calC$.
It is clear that $C^L_X\neq C^R_X$ for $X\in\calN^L\cap\calN^R$.
Consider the graph $G$ with vertex set $\calC$ and edges connecting $C^L_X$ and $C^R_X$ for all $X\in\calN^L\cap\calN^R$.
We are to show that $G$ is planar.

We construct sets $V\subset E\subset\bigcup\calN$ with the following properties:
\begin{enumerate}
\item $V$ is a finite subset of $\bigcup\calC$,
\item $E$ is a finite union of arcs with endpoints in $V$, pairwise disjoint outside of $V$,
\item $E\cap C$ is arc-connected for every component $C\in\calC$,
\item every $X\in\calN^L\cap\calN^R$ contains an arc in $E$ between $C^L_X$ and $C^R_X$.
\end{enumerate}
The construction proceeds by induction on the members of $\calN$.
We start with $V$ containing one (arbitrary) point in each member of $\calC$, and with $E=V$.
They clearly satisfy (i)--(iii).
Then, for each $X\in\calN^L\cap\calN^R$, we enlarge $V$ and $E$ to satisfy the condition (iv) for $X$, as follows.
Since $C^L_X\cup X\cup C^R_X$ is arc-connected, there is an arc $A\subset C^L_X\cup X\cup C^R_X$ between $E\cap C^L_X$ and $E\cap C^R_X$.
We can furthermore assume that $A\cap E\cap C^L_X=\{v^L\}$ and $A\cap E\cap C^R_X=\{v^R\}$.
This implies that $A\cap E=\{v^L,v^R\}$, as $X\setminus(C^L_X\cup C^R_X)$ is disjoint from any member of $\calN\setminus\{X\}$.
We add $v^L$ and $v^R$ to $V$ and $A$ to $E$.
After processing all $X\in\calN^L\cap\calN^R$, the resulting sets $V$ and $E$ satisfy (i)--(iv).

We have thus obtained a planar representation of a graph $H$ with vertex set $V$ and edge set consisting of maximal arcs in $E$ internally disjoint from $V$.
It follows from (iii) that the subgraph of $H$ induced on $V\cap C$ is connected for every $C\in\calC$.
Consider the graph obtained from $H$ by contracting $V\cap C$ for every $C\in\calC$.
Its vertices represent members of $\calC$, and by (iv), its edges connect the vertices representing $C^L_X$ and $C^R_X$ for all $X\in\calN^L\cap\calN^R$.
Hence this graph is isomorphic to $G$.
This shows that $G$ is planar, as a contraction of a planar graph is planar.

Since $G$ is planar, there is a proper coloring $\phi$ of the vertices of $G$ with four colors $\{1,2,3,4\}$.
Choose a coloring $\psi\colon\calN\to\{1,2,3,4\}$ so that
\begin{itemize}
\item if $X\in\calN^L$, then $\psi(X)=\phi(C^L_X)$,
\item if $X\in\calN^R$, then $\psi(X)\neq\phi(C^R_X)$.
\end{itemize}
Such a coloring exists, because $\phi(C^L_X)\neq\phi(C^R_X)$ for any $X\in\calN^L\cap\calN^R$.
To see that $\psi$ is a proper coloring of $\calN$, consider some $X,Y\in\calN$ such that $X\prec Y$ and $X\cap Y\neq\emptyset$.
Since $\phi^L(X)=\phi^L(Y)$ and $\phi^R(X)=\phi^R(Y)$ we have $\leftclip(X)\cap\leftclip(Y)=\emptyset$ and $\rightclip(X)\cap\rightclip(Y)=\emptyset$. 
Therefore,
\begin{equation*}
(X\setminus\rightclip(X))\cap(Y\setminus\leftclip(Y))=X\cap Y\neq\emptyset.
\end{equation*}
In particular, we have $X\in\calN^R$ and $Y\in\calN^L$.
The set $X\cap Y$ is arc-connected, so $C^R_X=C^L_Y$.
This yields $\psi(X)\neq\phi(C^R_X)=\phi(C^L_Y)=\psi(Y)$.
\end{proof}

\bibliographystyle{plain}
\bibliography{suk}

\end{document}